%% file: main.tex
\documentclass{article}

\usepackage{microtype}
\usepackage{graphicx}
\usepackage{subfigure}
\usepackage{booktabs} 

\usepackage{hyperref}

\usepackage{multirow}

\usepackage{icml2023}

\usepackage{amsmath}
\usepackage{amssymb}
\usepackage{mathtools}
\usepackage{amsthm}

\usepackage[normalem]{ulem}
\newcommand{\stkout}[1]{\ifmmode\text{\sout{\ensuremath{#1}}}\else\sout{#1}\fi}

\usepackage[capitalize,noabbrev]{cleveref}

\newcommand{\vb}{{\mathbf{b}}}
\newcommand{\vc}{{\mathbf{c}}}

\newcommand{\vf}{{\mathbf{f}}}

\newcommand{\vs}{{\mathbf{s}}}

\newcommand{\vx}{{\mathbf{x}}}
\newcommand{\vy}{{\mathbf{y}}}
\newcommand{\vz}{{\mathbf{z}}}

\newcommand{\A}{\mathbf{A}}

\newcommand{\V}{\mathbf{V}}
\newcommand{\X}{\mathbf{X}}

\newcommand{\Y}{\mathbf{Y}}

\newcommand{\J}{\mathbf{J}}

\newcommand{\f}{\mathbf{f}}

\newcommand{\0}{\mathbf{0}}
\newcommand{\1}{\mathbf{1}}
  
\newcommand{\sign}{\mathbf{sign}}

\newcommand{\bF}{\mathbf{F}}
\newcommand{\bbE}{\mathbb{E}}
\newcommand{\bbR}{\mathbb{R}}

\newcommand{\I}{\mathbf{I}}
\newcommand{\W}{\mathbf{W}}

\newcommand{\hE}{\mathcal{E}}
\newcommand{\hN}{\mathcal{N}}
\newcommand{\hG}{\mathcal{G}}

\newcommand{\dom}{{\mathrm{dom}}} 
\newcommand{\prox}{{\mathbf{Prox}}} 
\newcommand{\dist}{{\mathbf{dist}}} 
\DeclareMathOperator*{\argmin}{arg\,min} 

\usepackage{color}

\theoremstyle{plain}
\newtheorem{theorem}{Theorem}

\newtheorem{lemma}[theorem]{Lemma}
\newtheorem{corollary}[theorem]{Corollary}
\theoremstyle{definition}
\newtheorem{definition}[theorem]{Definition} 
\newtheorem{assumption}{Assumption}
\theoremstyle{remark}
\newtheorem{remark}{Remark}

\usepackage{enumitem}
\setlist{nolistsep}

\usepackage[textsize=tiny]{todonotes}

\icmltitlerunning{Decentralized Proximal Stochastic Gradient Method for Nonconvex Composite Problems}

\begin{document}

\twocolumn[
\icmltitle{Compressed Decentralized Proximal Stochastic Gradient Method for Nonconvex Composite Problems with Heterogeneous Data}



\icmlsetsymbol{equal}{*}

\begin{icmlauthorlist}
\icmlauthor{Yonggui Yan}{rpi}
\icmlauthor{Jie Chen}{ibm1}
\icmlauthor{Pin-Yu Chen}{ibm2}
\icmlauthor{Xiaodong Cui}{ibm2}
\icmlauthor{Songtao  Lu}{ibm2}
\icmlauthor{Yangyang Xu}{rpi}  
\end{icmlauthorlist}

\icmlaffiliation{rpi}{Department of Mathematical Sciences, Rensselaer Polytechnic Institute, Troy, NY, USA}
\icmlaffiliation{ibm1}{MIT IBM-Watson AI Lab, IBM Research, Cambridge, MA, USA}
\icmlaffiliation{ibm2}{Thomas J. Watson Research Center, IBM Research, Yorktown Heights, NY, USA}

\icmlcorrespondingauthor{Yangyang Xu}{xuy21@rpi.edu}
\icmlcorrespondingauthor{Yonggui Yan}{yany4@rpi.edu}

\icmlkeywords{Machine Learning, ICML}

\vskip 0.3in
]



\printAffiliationsAndNotice{}  

\begin{abstract} 
We first propose a decentralized proximal stochastic gradient tracking method (DProxSGT) for nonconvex stochastic composite problems, with data heterogeneously distributed on multiple workers in a decentralized connected network. To save communication cost, we then extend DProxSGT to a compressed method by compressing the communicated information. Both methods need only $\mathcal{O}(1)$ samples per worker for each proximal update, which is important to achieve good generalization performance on training deep neural networks. With a smoothness condition on the expected loss function (but not on each sample function), the proposed methods can achieve an optimal sample complexity result to produce a near-stationary point. Numerical experiments on training neural networks demonstrate the significantly better generalization performance of our methods over large-batch training methods and momentum variance-reduction methods and also, the ability of handling heterogeneous data by the gradient tracking scheme.
\end{abstract}

\input{problem}
\input{algorithms}

\input{analysis} 
\input{experiments}
\section{Conclusion}
We have proposed two decentralized proximal stochastic gradient methods, DProxSGT and CDProxSGT, for nonconvex composite problems with data heterogeneously distributed on the computing nodes of a connected graph. CDProxSGT is an extension of DProxSGT by applying compressions on the communicated model parameter and gradient information. Both methods need only a single or $\mathcal{O}(1)$ samples for each update, which is important to yield good generalization performance on training deep neural networks. The gradient tracking is used in both methods to address data heterogeneity. An $\mathcal{O}\left( \frac{1}{ \epsilon^4}\right)$ sample complexity and communication complexity is established to both methods to produce an expected $\epsilon$-stationary solution. 
Numerical experiments on training neural networks demonstrate the good generalization performance and the ability of the proposed methods on handling heterogeneous data. 


\bibliography{reference}
\bibliographystyle{icml2023}

\newpage
\appendix
\onecolumn


\input{appendix_proof_sect1}

\input{appendix_proof_sect2}

\input{appendix_proof_sect3}

\input{appendix_experiment} 
\end{document}

%% file: problem.tex
\section{Introduction} 
In this paper, we consider to solve nonconvex stochastic composite problems in a decentralized setting:
 \vspace{-0.1cm}
\begin{equation}\label{eq:problem_original}
 \begin{aligned}
&  \min_{\vx\in\bbR^d} \phi(\vx) = f(\vx) + r(\vx),\\[-0.1cm] 
& \text{with } f(\vx)=\frac{1}{n}\sum_{i=1}^n f_i(\vx), f_i(\vx)\!=\!\bbE_{\xi_i \sim \mathcal{D}_i}[F_i(\vx,\xi_i)].
    \end{aligned}
 \vspace{-0.1cm}   
\end{equation}
Here, $\{\mathcal{D}_i\}_{i=1}^n$ are possibly \emph{non-i.i.d data} distributions on $n$ machines/workers that can be viewed as nodes of a connected graph $\hG$, and each $F_i(\cdot, \xi_i)$ can only be accessed by the $i$-th worker. 
We are interested in problems 
that satisfy the following structural assumption. 
\begin{assumption}[Problem structure] \label{assu:prob}
We assume that 
\vspace{-1.5mm}
\begin{itemize}  
\item[(i)] $r$ is closed convex and possibly nondifferentiable.
\item[(ii)] Each $f_i$ is $L$-smooth in $\dom(r)$, i.e., $\|\nabla f_i(\vx) - \nabla f_i(\vy)\| \le L \|\vx- \vy\|$, for any $\vx, \vy\in\dom(r)$. 
\item[(iii)] $\phi$ is lower bounded, i.e.,  $\phi^* \triangleq \min_\vx \phi(\vx) > -\infty$.
\end{itemize}
\vspace{-2mm}
\end{assumption}


Let $\hN=\{1, 2, \ldots, n\}$ be the set of nodes of $\hG$ and $\hE$ the set of edges. 
For each $i\in\hN$, denote $\hN_i$ as the neighbors of worker $i$ and itself, i.e., $\hN_i = \{j: (i,j) \in \hE\}\cup \{i\}$. Every worker can only communicate with its neighbors. To solve \eqref{eq:problem_original} collaboratively, each worker $i$ maintains a copy, denoted as $\vx_i$, of the variable $\vx$. With these notations, 
 %
\eqref{eq:problem_original} can be formulated equivalently to 
\vspace{-0.1cm}
{\begin{align}\label{eq:decentralized_problem} 
\begin{split}
\min_{\X \in \bbR^{d\times n}}  &  \frac{1}{n}\sum_{i=1}^n \phi_i(\vx_i), \text{with }\phi_i(\vx_i) \triangleq f_i(\vx_i) + r(\vx_i), \\ 
 \mbox{s.t.  } \quad & \vx_i=\vx_j, \forall\, j\in \hN_i, \forall\, i = 1,\ldots, n.
\end{split} 
\end{align}}
\vspace{-0.5cm}

Problems with a \emph{nonsmooth} regularizer, i.e., in the form of \eqref{eq:problem_original}, appear in many applications such as $\ell_1$-regularized signal recovery \cite{eldar2014phase,duchi2019solving}, online nonnegative matrix factorization \cite{guan2012online}, and training sparse neural networks \cite{scardapane2017group, yang2020proxsgd}. When data involved in these applications are distributed onto (or collected by workers on) a decentralized network, it necessitates the design of decentralized algorithms. 

Although decentralized optimization has attracted a lot of research interests in recent years, most existing works focus on strongly convex problems \cite{scaman2017optimal,  koloskova2019decentralized} or convex problems \cite{6426375,taheri2020quantized} or smooth nonconvex problems \cite{bianchi2012convergence, di2016next, wai2017decentralized, lian2017can,zeng2018nonconvex}.
Few works have studied \emph{nonsmooth nonconvex} decentralized \emph{stochastic} optimization like \eqref{eq:decentralized_problem} that we consider. \cite{chen2021distributed, xin2021stochastic, mancino2022proximal} are among the exceptions. However, they either require to take many data samples for each update or assume a so-called mean-squared smoothness condition, which is stronger than the smoothness condition in Assumption~\ref{assu:prob}(ii), in order to perform momentum-based variance-reduction step. Though these methods can have convergence (rate) guarantee, they often yield poor generalization performance on training deep neural networks, as demonstrated in  \cite{lecun2012efficient, keskar2016large} for large-batch training methods and in our numerical experiments for momentum variance-reduction methods.

On the other side, many distributed optimization methods \cite{shamir2014distributed,lian2017can,wang2018cooperative} 
often assume that the data are i.i.d across the workers. 
However, this assumption does not hold in many real-world scenarios, for instance, due to data privacy issue that local data has to stay on-premise. 
Data heterogeneity can result in significant degradation of the performance by these methods. 
Though some papers do not assume i.i.d. data, they require certain data similarity, such as bounded stochastic gradients \cite{koloskova2019decentralized,koloskova2019decentralized-b, taheri2020quantized} and bounded gradient dissimilarity \cite{ tang2018communication,assran2019stochastic, tang2019deepsqueeze, vogels2020practical}.

To address the critical practical issues mentioned above, we propose a decentralized proximal stochastic gradient tracking method that needs only a single or $O(1)$ data samples (per worker) for each update. With no assumption on data similarity, it can still achieve the optimal convergence rate on solving problems satisfying conditions in Assumption~\ref{assu:prob} and yield good generalization performance. In addition, to reduce communication cost, we give a compressed version of the proposed algorithm, by performing compression on the communicated information. The compressed algorithm can inherit the benefits of its non-compressed counterpart.

\subsection{Our Contributions}

Our contributions are three-fold. First, we propose two decentralized algorithms, one without compression (named DProxSGT) and the other with compression (named CDProxSGT), for solving \emph{decentralized nonconvex nonsmooth stochastic} problems. Different from existing methods, e.g., \cite{xin2021stochastic, wang2021distributed, mancino2022proximal}, which need a very large batchsize and/or perform momentum-based variance reduction to handle the challenge from the nonsmooth term, DProxSGT needs only $\mathcal{O}(1)$ data samples for each update, without performing variance reduction. The use of a small batch and a standard proximal gradient update enables our method to achieve significantly better generalization performance over the existing methods, as we demonstrate on training neural networks. To the best of our knowledge, CDProxSGT is the first decentralized algorithm that applies a compression scheme for solving nonconvex nonsmooth stochastic problems, and it inherits the advantages of the non-compressed method DProxSGT. Even applied to the special class of smooth nonconvex problems, CDProxSGT can perform significantly better over state-of-the-art methods, in terms of generalization and handling data heterogeneity.

Second, we establish an optimal sample complexity result of DProxSGT, which matches the lower bound result in \cite{arjevani2022lower} in terms of the dependence on a target tolerance $\epsilon$, to produce an $\epsilon$-stationary solution. Due to the coexistence of nonconvexity, nonsmoothness, big stochasticity variance (due to the small batch and no use of variance reduction for better generalization), and decentralization, the analysis is highly non-trivial. We employ the tool of Moreau envelope and construct a decreasing Lyapunov function by carefully controlling the errors introduced by stochasticity and decentralization. 

Third, we establish the iteration complexity result of the proposed compressed method CDProxSGT, which is in the same order as that for DProxSGT and thus also optimal in terms of the dependence on a target tolerance. The analysis builds on that of DProxSGT but is more challenging due to the additional compression error and the use of gradient tracking. Nevertheless, we obtain our results by making the same (or even weaker) assumptions as those assumed by state-of-the-art methods \cite{koloskova2019decentralized-b, zhao2022beer}.

\subsection{Notation}\label{sec:notation}
For any vector $\vx\in\bbR^{d}$, we use $\|\vx\|$ for the $\ell_2$ norm. For any matrix $\A$, $\|\A\|$ denotes the Frobenius norm and $\|\A\|_2$ the spectral norm.
$\X = [\vx_1,\vx_2,\ldots,\vx_n]\in\bbR^{d\times n}$ concatinates all local variables. The superscript $^t$ will be used for iteration or communication.
$\nabla F_i(\vx_i^t,\xi_i^t)$ denotes a local stochastic gradient of $F_i$ at $\vx_i^t$ with a random sample $\xi_i^t$. The column concatenation of $\{\nabla F_i(\vx_i^t,\xi_i^t)\}$ is denoted as 
\vspace{-0.1cm}
\begin{equation*}
\nabla \bF^t = \nabla \bF(\X^t,\Xi^t) = [ \nabla F_1(\vx_1^t,\xi_1^t),\ldots, \nabla F_n(\vx_n^t,\xi_n^t)],\vspace{-0.1cm}
\end{equation*}
where  $\Xi^t = [\xi_1^t,\xi_2^t,\ldots,\xi_n^t]$.
Similarly, we denote 
\vspace{-0.1cm}
\begin{equation*} 
\nabla \f^t  
= [ \nabla f_1(\vx_1^t ),\ldots, \nabla f_n(\vx_n^t )].\vspace{-0.1cm}
\end{equation*} 
For any $\X \in \bbR^{d\times n}$, 
we define 
\vspace{-0.1cm}
\begin{equation*} \bar{\vx} = \textstyle\frac{1}{n}\X\1, \quad \overline{\X} = \X\J = \bar{\vx}\1^\top,\quad \X_\perp = \X(\I - \J), \vspace{-0.1cm}
\end{equation*} 
where $\1$ is the all-one vector, and $\J = \frac{\1\1^\top}{n}$ is the averaging matrix.
Similarly, we define the mean vectors 
\vspace{-0.1cm}
\begin{equation*} 
\overline{\nabla} \bF^t =  \textstyle\frac{1}{n} \bF^t \1,\ \overline{\nabla} \f^t = \textstyle\frac{1}{n} \f^t \1.\vspace{-0.1cm}
\end{equation*} 
We will use $\bbE_t$ for the expectation about the random samples $\Xi^t$ at the $t$th iteration and $\bbE$ for the full expectation. $\bbE_Q$ denotes the expectation about a stochastic compressor $Q$.




\section{Related Works}
The literature of decentralized optimization has been growing vastly. To exhaust the literature is impossible. Below we review existing works on decentralized algorithms for solving nonconvex problems, with or without using a compression technique. For ease of understanding the difference of our methods from existing ones, 
we compare to a few relevant methods in Table \ref{tab:method_compare}.

\begin{table*}[t]\label{tab:method_compare}
\caption{Comparison between our methods and some relevant methods: ProxGT-SA and ProxGT-SR-O  in \cite{xin2021stochastic}, DEEPSTORM \cite{mancino2022proximal},  ChocoSGD \cite{koloskova2019decentralized-b}, and  BEER \cite{zhao2022beer}. We 
use ``CMP'' to represent whether compression is performed by a method. 
GRADIENTS represents additional assumptions on the stochastic gradients in addition to those made in Assumption \ref{assu:stoc_grad}. 
SMOOTHNESS represents the smoothness condition, where ``mean-squared'' means $\bbE_{\xi_i}[\|\nabla F_i(\vx; \xi_i) - \nabla F_i(\vy; \xi_i)\|^2]\le L^2\|\vx-\vy\|^2$ that is stronger than the $L$-smoothness of $f_i$.
BS is the required batchsize to get an $\epsilon$-stationary solution. VR and MMT represent whether the variance reduction or momentum are used. Large batchsize and/or momentum variance reduction can degrade the generalization performance, as we demonstrate 
in numerical experiments. 
}
\label{sample-table}
\begin{center}
\begin{small}
\begin{sc}
\begin{tabular}{lccccc}
\toprule
 Methods & CMP & $r\not\equiv 0$  & GRADIENTS  & SMOOTHNESS & (BS, VR, MMT)  \\
\midrule
 ProxGT-SA & No& Yes & No & $f_i$ is smooth  &  \big($\mathcal{O}(\frac{1}{\epsilon^2})$, No , No\big) \\[0.1cm]
  ProxGT-SR-O   & No & Yes  & No &  mean-squared &  \big($\mathcal{O}(\frac{1}{\epsilon})$, Yes, No\big)  \\[0.1cm]
  DEEPSTORM  & No & Yes & No &  mean-squared & ($\mathcal{O}(1)$, Yes, Yes)  \\
  \textbf{DProxSGT (this paper)}   & No & Yes & No &  $f_i$ is smooth & ($\mathcal{O}(1)$, No, No)  \\
\midrule
   ChocoSGD  & Yes& No & $\bbE_{\xi}[\|\nabla F_i(\vx,\xi_i)\|^2]\leq G^2$   & $f_i$ is smooth & ($\mathcal{O}(1)$, No, No)  \\
   BEER  & Yes & No & No  & $f$ is smooth &  \big($\mathcal{O}(\frac{1}{\epsilon^2})$, No, No\big) \\[0.1cm]
   \textbf{CDProxSGT (this paper)}  & Yes & Yes & No & $f_i$ is smooth & ($\mathcal{O}(1)$, No, No)  \\
\bottomrule
\end{tabular}
\end{sc}
\end{small}
\end{center}
\vskip -0.1in
\end{table*}

\subsection{Non-compressed Decentralized Methods}

For nonconvex decentralized problems with a nonsmooth regularizer, a lot of deterministic decentralized methods have been studied, e.g., \cite{di2016next, wai2017decentralized, zeng2018nonconvex, chen2021distributed, scutari2019distributed}.
When only stochastic gradient is available, a majority of existing works focus on smooth cases without a regularizer or a hard constraint, such as \cite{lian2017can, assran2019stochastic, tang2018d}, 
gradient tracking based methods \cite{lu2019gnsd,zhang2019decentralized, koloskova2021improved}, 
and momentum-based variance reduction methods \cite{xin2021hybrid, zhang2021gt}. 
Several works such as  \cite{bianchi2012convergence, wang2021distributed, xin2021stochastic, mancino2022proximal} have studied stochastic decentralized methods for problems with a nonsmooth term $r$. 
However, they either consider some special $r$ or require a large batch size. \cite{bianchi2012convergence} considers 
the case where $r$ is an indicator function of a compact convex set. Also, it 
requires bounded stochastic gradients. 
\cite{wang2021distributed} focuses on problems with a polyhedral $r$, and it 
requires a large batch size of $\mathcal{O}(\frac{1}{\epsilon})$ to produce an (expected) $\epsilon$-stationary point.
\cite{xin2021stochastic, mancino2022proximal} are the most closely related to our methods. To produce an (expected) $\epsilon$-stationary point, the methods in \cite{xin2021stochastic} require a large batch size, either $\mathcal{O}(\frac{1}{\epsilon^2})$ or $\mathcal{O}(\frac{1}{\epsilon})$ if variance reduction is applied. 
The method in \cite{mancino2022proximal} requires only $\mathcal{O}(1)$ samples for each update by taking a momentum-type variance reduction scheme. However, in order to reduce variance, it needs a stronger mean-squared smoothness assumption. In addition, the momentum variance reduction step can often hurt the generalization performance on training complex neural networks, as we will demonstrate in our numerical experiments. 

\subsection{Compressed Distributed Methods}

Communication efficiency is a crucial factor when designing a distributed optimization strategy. The current machine learning paradigm oftentimes resorts to models with a large number of parameters, which indicates a high communication cost when the models or gradients are transferred from workers to the parameter server or among workers. This may incur significant latency in training. Hence, communication-efficient algorithms by model or gradient compression have been actively sought.

Two major groups of compression operators are quantization and sparsification. The quantization approaches include 1-bit SGD \cite{seide20141}, SignSGD \cite{bernstein2018signsgd}, QSGD \cite{alistarh2017qsgd}, TernGrad \cite{wen2017terngrad}. The sparsification approaches include Random-$k$ \cite{stich2018sparsified}, Top-$k$ \cite{aji2017sparse}, Threshold-$v$ \cite{dutta2019discrepancy} and ScaleCom \cite{chen2020scalecom}. Direct compression may slow down the convergence especially when compression ratio is high. Error compensation or error-feedback can mitigate the effect by saving the compression error in one communication step and compensating it in the next communication step before another compression \cite{seide20141}. These compression operators are first designed to compress the gradients in the centralized setting \cite{tang2019DoubleSqueeze,karimireddy2019error}.
 
The compression can also be applied to the decentralized setting for smooth problems, i.e., \eqref{eq:decentralized_problem} with $r=0$. \cite{tang2019deepsqueeze} applies the compression with error compensation to the  communication of model parameters in the decentralized seeting.
Choco-Gossip \cite{koloskova2019decentralized} is another communication way to mitigate the slow down effect from compression. It does not compress the model parameters but a residue between  model parameters and its estimation. Choco-SGD uses Choco-Gossip to solve \eqref{eq:decentralized_problem}. BEER \cite{zhao2022beer} includes gradient tracking and compresses both tracked stochastic gradients and model parameters in each iteration by the Choco-Gossip.
BEER needs a large batchsize of $\mathcal{O}(\frac{1}{\epsilon^2})$ in order to produce an $\epsilon$-stationary solution.
DoCoM-SGT\cite{DBLP:journals/corr/abs-2202-00255} does similar updates as BEER but with a momentum term for the update of the tracked gradients, and it only needs an $\mathcal{O}(1)$ batchsize. 

Our proposed CDProxSGT is for solving decentralized problems in the form of \eqref{eq:decentralized_problem} with a nonsmooth $r(\vx)$. To the best of our knowledge, CDProxSGT is the first compressed decentralized method for nonsmooth nonconvex problems without the use of a large batchsize, and it can achieve an optimal sample complexity without the assumption of data similarity or gradient boundedness. 



%% file: algorithms.tex
\section{Decentralized Algorithms}\label{sec:alg}

In this section, we give our decentralized algorithms for solving \eqref{eq:decentralized_problem} or equivalently \eqref{eq:problem_original}. To perform neighbor communications, we introduce a mixing (or gossip) matrix $\W$ that satisfies the following standard assumption.
\begin{assumption}[Mixing matrix] \label{assu:mix_matrix} We choose a mixing matrix $\W$ such that
\vspace{-1.5mm}
\begin{enumerate}
\item [(i)] $\W$ is doubly stochastic: $\W\1 = \1$ and $\1^\top \W = \1^\top$;
\item [(ii)] $\W_{ij} = 0$ if $i$ and $j$ are not neighbors to each other;
\item [(iii)] $\mathrm{Null}(\W-\I) = \mathrm{span}\{\1\}$ and $\rho \triangleq \|\W - \J\|_2 < 1$.
\end{enumerate}
\vspace{-2mm}
\end{assumption}
The condition in (ii) above is enforced so that \emph{direct} communications can be made only if two nodes (or workers) are immediate (or 1-hop) neighbors of each other. The condition in (iii) can hold if the graph $\hG$ is connected. The assumption $\rho < 1$ is critical to ensure contraction of consensus error. 

The value of $\rho$ depends on the graph topology.   
\cite{koloskova2019decentralized} gives three commonly used examples: when uniform weights are used between nodes, $\W = \J$ and $\rho = 0$ for a fully-connected graph (in which case, our algorithms will reduce to centralized methods), $1-\rho = \Theta(\frac{1}{n})$ for a 2d torus grid graph where every node has 4 neighbors, and $1-\rho = \Theta(\frac{1}{n^2})$ for a ring-structured graph. 
More examples can be found in \cite{nedic2018network}.



\subsection{Non-compreseed Method}

With the mixing matrix $\W$, we propose a decentralized proximal stochastic gradient method with gradient tracking (DProxSGT) for \eqref{eq:decentralized_problem}. The pseudocode is shown in Algorithm~\ref{alg:DProxSGT}. 
In every iteration $t$, each node $i$ first computes a local stochastic gradient $\nabla F_i(\vx_i^{t},\xi_i^{t})$ by taking a sample $\xi_i^{t}$ from its local data distribution $\mathcal{D}_i$, then performs gradient tracking in \eqref{eq:y_half_update} and neighbor communications of the tracked gradient in \eqref{eq:y_update}, and finally takes a proximal gradient step in \eqref{eq:x_half_update} and mixes the model parameter with its neighbors in \eqref{eq:x_1_update}.

\begin{algorithm}[tb]
   \caption{DProxSGT}\label{alg:DProxSGT}
\begin{algorithmic}
   \small{ \STATE Initialize $\vx_i^{0}$ and set $\vy_i^{-1}=\0$, $\nabla F_i(\vx_i^{-1},\xi_i^{-1}) =\0$,   $\forall i\in\hN$. 
   \FOR{$t=0, 1, 2, \ldots, T-1$} 
   \STATE \hspace{-0.1cm}\textbf{all} nodes $i=1, 2, \ldots, n$ do the updates \textbf{in parallel:}
   \STATE obtain one random sample $\xi_i^t$, compute a stochastic gradient $\nabla F_i(\vx_i^{t},\xi_i^{t})$, and perform 
   \vspace{-0.2cm}
       \begin{gather}
            	\vy_i^{t-\frac{1}{2}} =  \vy_i^{t-1}  +  \nabla F_i(\vx_i^{t},\xi_i^{t}) - \nabla F_i(\vx_i^{t-1},\xi_i^{t-1}),\label{eq:y_half_update}
            	\\
            	\vy_i^t =  \textstyle \sum_{j=1}^n \W_{ji} \vy_j^{t-\frac{1}{2}},\label{eq:y_update}\\
            	\vx_i^{t+\frac{1}{2}} =\prox_{\eta r} \left(\vx_i^t - \eta \vy_i^{t}\right),\label{eq:x_half_update}
            	\\
            	\vx_i^{t+1} =  \textstyle \sum_{j=1}^n \W_{ji}\vx_j^{t+\frac{1}{2}}. \label{eq:x_1_update}
          \vspace{-0.2cm}   
    	\end{gather}   
   \ENDFOR}
\end{algorithmic}
\end{algorithm}
\vspace{-0.1mm}

Note that for simplicity, we take only one random sample $\xi_i^{t}$ in Algorithm \ref{alg:DProxSGT} but in general, a mini-batch of random samples can be taken, and all theoretical results that we will establish in the next section still hold. We emphasize that we need only $\mathcal{O}(1)$ samples for each update. This is different from ProxGT-SA in \cite{xin2021stochastic}, which shares a similar update formula as our algorithm but needs a very big batch of samples, as many as $\mathcal{O}(\frac{1}{\epsilon^2})$, where $\epsilon$ is a target tolerance. A small-batch training can usually generalize better than a big-batch one \cite{lecun2012efficient, keskar2016large} on training large-scale deep learning models. Throughout the paper, we make the following standard assumption on the stochastic gradients.

\begin{assumption}[Stochastic gradients] \label{assu:stoc_grad}
We assume that 
\vspace{-1.5mm}
\begin{itemize}
    \item[(i)] The random samples $\{\xi_i^t\}_{i\in \hN, t\ge 0}$ are independent. 
    \item[(ii)] 
    There exists a finite number $\sigma\ge0$ such that for any $i\in \hN$ and $\vx_i\in\dom(r)$, 
    \begin{gather*}  
    \bbE_{\xi_i}[\nabla F_i(\vx_i,\xi_i)]  = \nabla f_i(\vx_i),\\ 
    \bbE_{\xi_i}[\|\nabla F_i(\vx_i,\xi_i)-\nabla f_i(\vx_i)\|^2] \leq \sigma^2.
     \end{gather*}  
\end{itemize}
\vspace{-2mm}
\end{assumption}



The gradient tracking step in \eqref{eq:y_half_update} 
is critical to handle heterogeneous data 
\cite{di2016next,nedic2017achieving,lu2019gnsd,pu2020distributed,sun2020improving,xin2021stochastic,song2021optimal,mancino2022proximal,zhao2022beer,DBLP:journals/corr/abs-2202-00255,song2022compressed}.
In a deterministic scenario where $\nabla f_i(\cdot)$ is used instead of $\nabla F_i(\cdot, \xi)$, for each $i$, the tracked gradient $\vy_i^t$ can converge to the gradient of the global function $\frac{1}{n}\sum_{i=1}^n f_i(\cdot)$ at $\bar\vx^t$, and thus all local updates move towards a direction to minimize the \emph{global} objective. When stochastic gradients are used, the gradient tracking can play a similar role and make $\vy_i^t$ approach to the stochastic gradient of the global function. 
With this nice property of gradient tracking, we can guarantee convergence without strong assumptions that are made in existing works, such as bounded gradients  \cite{koloskova2019decentralized,koloskova2019decentralized-b, taheri2020quantized, singh2021squarm}  and bounded data similarity over nodes \cite{lian2017can, tang2018communication, tang2019deepsqueeze, vogels2020practical, wang2021error}.

\subsection{Compressed Method}
In DProxSGT, each worker needs to communicate both the model parameter and tracked stochastic gradient with its neighbors at every iteration. Communications have become a bottleneck for distributed training on GPUs. In order to save the communication cost, we further propose a compressed version of DProxSGT, named CDProxSGT. The pseudocode is shown in Algorithm \ref{alg:CDProxSGT}, where $Q_\vx$ and $Q_\vy$ are two compression operators.

\begin{algorithm}[tb]
\caption{CDProxSGT}\label{alg:CDProxSGT}
\begin{algorithmic}
  \small{\STATE Initialize  $\vx_i^{0}$; set 
  $\vy_i^{-1}=\underline\vy_i^{-1}=\nabla F_i(\vx_i^{-1}, \xi_i^{-1})=\underline\vx_i^{0} =\0$, $\forall i\in\hN$. 
  \FOR{$t=0, 1, 2, \ldots, T-1$}
     \STATE \hspace{-0.1cm}\textbf{all} nodes $i=1, 2, \ldots, n$ do the updates \textbf{in parallel:}
     \vspace{-0.2cm}
         \begin{gather}
    	\vy_i^{t-\frac{1}{2}} =  \vy_i^{t-1} + \nabla F_i(\vx_i^{t},\xi_i^{t}) - \nabla F_i(\vx_i^{t-1},\xi_i^{t-1}),\label{eq:alg3_1}\\
        \underline\vy_i^{t} = \underline\vy_i^{t-1} + Q_\vy\big[\vy_i^{t-\frac{1}{2}} - \underline\vy_i^{t-1}\big], \label{eq:alg3_2}\\
        \vy_i^{t} = \vy_i^{t-\frac{1}{2}} +\gamma_y  \left(\textstyle \sum_{j=1}^n \W_{ji} \underline\vy_j^{t}-\underline\vy_i^{t}\right), \label{eq:alg3_3}\\
    	\vx_i^{t+\frac{1}{2}} =\prox_{\eta r} \left(\vx_i^t - \eta \vy_i^{t}\right), \label{eq:alg3_4}\\
    	\underline\vx_i^{t+1} = \underline\vx_i^{t} + Q_\vx\big[\vx_i^{t+\frac{1}{2}} - \underline\vx_i^{t}\big], \label{eq:alg3_5}\\
    	\vx_i^{t+1} = \vx_i^{t+\frac{1}{2}}+\gamma_x\Big(\textstyle  \overset{n}{\underset{j=1}\sum} \W_{ji} \underline\vx_j^{t+1}-\underline\vx_i^{t+1}\Big).\label{eq:alg3_6}
     \vspace{-0.2cm}
     \end{gather}    
   \ENDFOR}
\end{algorithmic}
\end{algorithm}
\vspace{-0.1mm}

In Algorithm \ref{alg:CDProxSGT}, each node communicates the non-compressed vectors $\underline\vy_i^{t}$ and $\underline\vx_i^{t+1}$  with its neighbors in \eqref{eq:alg3_3} and \eqref{eq:alg3_6}. We write it in this way for ease of read and analysis. For efficient and \emph{equivalent} implementation, 
we do not communicate 
$\underline\vy_i^{t}$ and $\underline\vx_i^{t+1}$ directly but the compressed residues  $Q_\vy\big[\vy_i^{t-\frac{1}{2}} - \underline\vy_i^{t-1}\big]$ and $Q_\vx\big[\vx_i^{t+\frac{1}{2}} - \underline\vx_i^{t}\big]$, explained as follows. 
Besides $\vy_i^{t-1}$, $\vx_i^{t}$, $\underline\vy_i^{t-1}$ and $\underline\vx_i^{t}$, each node also stores $\vz_i^{t-1}$ and $\vs_i^{t}  $ which record $\sum_{j=1}^n \W_{ji} \underline\vy_i^{t-1}$ and $\sum_{j=1}^n \W_{ji} \underline\vx_i^{t}$. For the gradient communication, each node $i$ initializes $\vz_i^{-1} = \0$, and then at each iteration $t$, 
after receiving $Q_\vy\big[\vy_j^{t-\frac{1}{2}} - \underline\vy_j^{t-1}\big]$ from its neighbors, it updates $\underline\vy_i^{t}$ by \eqref{eq:alg3_2}, and $\vz_i^{t}$ and $\vy_i^t$ by 
\vspace{-0.2cm} 
\begin{align*}
  \vz_i^{t} =&~ \textstyle \vz_i^{t-1} + \sum_{j=1}^n \W_{ji} Q_\vy\big[\vy_j^{t-\frac{1}{2}} - \underline\vy_j^{t-1}\big], \\ 
\vy_i^{t} =&~ \textstyle \vy_i^{t-\frac{1}{2}} +\gamma_y \big(\vz_i^{t}-\underline\vy_i^{t}\big).\vspace{-0.2cm}
\end{align*}
From the initialization and the updates of $\underline\vy_i^{t}$ and $\vz_i^{t}$, 
it always holds that
$\vz_i^{t}=\sum_{j=1}^n \W_{ji} \underline\vy_i^{t}$.
The model communication can be done efficiently in the same way. 

The compression operators $Q_\vx$ and $Q_\vy$ in Algorithm \ref{alg:CDProxSGT} can be different, but we assume that they both satisfy the following assumption. 
\begin{assumption} \label{assu:compressor}
There exists $\alpha \in [0,1)$ such that  
$$\bbE[\|\vx-Q[\vx]\|^2]\leq \alpha^2 \|\vx\|^2, \forall\, \vx\in\bbR^d,$$ 
for both $Q=Q_\vx$ and $Q=Q_\vy$.
\end{assumption}
The assumption on compression operators is standard and also made in \cite{koloskova2019decentralized-b,koloskova2019decentralized,zhao2022beer}. It is satisfied by the sparsification, such as Random-$k$ \cite{stich2018sparsified} and Top-$k$ \cite{aji2017sparse}. It can also be satisfied by rescaled quantizations.  For example, QSGD \cite{alistarh2017qsgd} compresses $\vx\in \bbR^d$ by $Q_{sqgd}(\vx) = \frac{\sign(\vx)\|\vx\|}{s} \lfloor s \frac{|\vx|}{\|\vx\|} + \xi \rfloor
$
where $\xi$ is uniformly distributed on $[0,1]^d$, $s$ is the parameter about compression level. Then $Q(\vx)= \frac{1}{\tau} Q_{sqgd} (\vx)$ with $\tau=(1+\min\{d/s^2, \sqrt{d}/s\})$ satisfies Assumption \ref{assu:compressor} with $\alpha^2=1-\frac{1}{\tau}$. More examples can be found in \cite{koloskova2019decentralized}.

Below, we make a couple of remarks to discuss the relations between Algorithm \ref{alg:DProxSGT} and Algorithm \ref{alg:CDProxSGT}.

\begin{remark}
When $Q_\vx$ and $Q_\vy$ are both identity operators, i.e., $Q_\vx[\vx] = \vx, Q_\vy[\vy] = \vy$, and $\gamma_x=\gamma_y=1$, in Algorithm \ref{alg:CDProxSGT}, CDProxSGT will reduce to DProxSGT. Hence, the latter can be viewed as a special case of the former. However, we will analyze them separately. Although the big-batch training method ProxGT-SA in \cite{xin2021stochastic} shares a similar update as the proposed DProxSGT, our analysis will be completely different and new, as we need only $\mathcal{O}(1)$ samples in each iteration in order to achieve better generalization performance. The analysis of CDProxSGT will be built on that of DProxSGT by carefully controlling the variance error of stochastic gradients and the consensus error, as well as the additional compression error. 
\end{remark}


\begin{remark}
When $Q_\vy$ and $Q_\vx$ are identity operators, $\underline\vy_i^{t} = \vy_i^{t-\frac{1}{2}}$ and $\underline\vx_i^{t+1} = \vx_i^{t+\frac{1}{2}}$ for each $i\in\hN$. Hence, in the compression case, $\underline\vy_i^{t}$ and $\underline\vx_i^{t+1}$ can be viewed as estimates of  $\vy_i^{t-\frac{1}{2}}$ and $\vx_i^{t+\frac{1}{2}}$. 
In addition, in a matrix format, we have from \eqref{eq:alg3_3} and \eqref{eq:alg3_6} that
 \begin{align} 
    \Y^{t+1} 
     =&~ \Y^{t+\frac{1}{2}}\widehat\W_y + \gamma_y\big(\underline\Y^{t+1}-\Y^{t+\frac{1}{2}}\big)(\W-\I), \label{eq:Y_hatW}\\
    \X^{t+1} 
    =&~ \X^{t+\frac{1}{2}}\widehat\W_x + \gamma_x(\underline\X^{t+1}-\X^{t+\frac{1}{2}})(\W-\I), \label{eq:compX_hatW}
\end{align}
where 
$\widehat\W_y = \gamma_y \W + (1-\gamma_y)\I,\ \widehat\W_x = \gamma_x \W + (1-\gamma_x)\I.$
When $\W$ satisfies the conditions (i)-(iii) in Assumption~\ref{assu:mix_matrix}, it can be easily shown that $\widehat\W_y$ and $\widehat\W_x$ also satisfy all three conditions. Indeed, we have
$$\widehat\rho_x \triangleq \|\widehat\W_x - \J\|_2 < 1,\quad 
\widehat\rho_y \triangleq \|\widehat\W_y - \J\|_2 < 1.$$
Thus we can view $\Y^{t+1}$ and $\X^{t+1}$ as the results of $\Y^{t+\frac{1}{2}}$ and $\X^{t+\frac{1}{2}}$ by one round of neighbor communication with mixing matrices $\widehat{\W}_y$ and $\widehat{\W}_x$, and the addition of the estimation error $\underline\Y^{t+1}-\Y^{t+\frac{1}{2}}$ and $\underline\X^{t+1}-\X^{t+\frac{1}{2}}$ after one round of neighbor communication.
\end{remark}
%
%



%% file: analysis.tex
\section{Convergence Analysis}

In this section, we analyze the convergence of the algorithms proposed in section~\ref{sec:alg}. Nonconvexity of the problem and stochasticity of the algorithms both raise difficulty on the analysis. In addition, the coexistence of the nonsmooth regularizer $r(\cdot)$ causes more significant challenges.  
To address these challenges, we employ a tool of the so-called Moreau envelope \cite{MR201952}, which has been commonly used for analyzing methods on solving nonsmooth weakly-convex problems.

\begin{definition}[Moreau envelope] Let $\psi$ be an $L$-weakly convex function, i.e., $\psi(\cdot) + \frac{L}{2}\|\cdot\|^2$ is convex. For $\lambda\in(0,\frac{1}{L})$, the Moreau envelope of $\psi$ is defined as
\vspace{-0.2cm}
 \begin{equation*} 
\psi_\lambda(\vx) = \min_\vy \textstyle \left\{\psi(\vy) + \frac{1}{2\lambda}\|\vy-\vx\|^2\right\},  \vspace{-0.2cm}
\end{equation*} 
and the unique minimizer is denoted as
\vspace{-0.2cm}
\begin{equation*}
 \prox_{\lambda \psi}(\vx)= \argmin_{\vy} \textstyle \left\{\psi(\vy)+\frac{1}{2\lambda} \|\vy-\vx\|^2\right\}.\vspace{-0.2cm}
 \end{equation*} 
\end{definition}


The Moreau envelope $\psi_\lambda$ has nice properties.
The result below can be found in  
  \cite{davis2019stochastic, nazari2020adaptive, xu2022distributed-SsGM}.
 \begin{lemma}\label{lem:xhat_x}
 For any function $\psi$, if it is $L$-weakly convex, then for any $\lambda \in (0, \frac{1}{L})$,  the Moreau envelope $\psi_\lambda$ is smooth with gradient given by 
$\nabla \psi_\lambda (\vx) = \lambda^{-1} (\vx-\widehat\vx),$
 where $\widehat\vx=\prox_{\lambda\psi}(\vx)$. 
Moreover, \vspace{-0.2cm}
\[
\|\vx-\widehat\vx\|=\lambda\|\nabla \psi_\lambda(\vx)\|, \quad 
\dist(\0, \partial \psi(\widehat \vx))\leq \|\nabla \psi_\lambda(\vx)\|.\vspace{-0.2cm}
\]
\end{lemma}

Lemma~\ref{lem:xhat_x}  implies that if $\|\nabla \psi_\lambda(\vx)\|$ is small, then $\widehat\vx$ is a near-stationary point of $\psi$ and $\vx$ is close to $\widehat\vx$. Hence, $\|\nabla \psi_\lambda(\vx)\|$ can be used as a valid measure of stationarity violation at $\vx$ for $\psi$. Based on this observation, we define the $\epsilon$-stationary solution below for the decentralized problem \eqref{eq:decentralized_problem}.


\begin{definition}[Expected $\epsilon$-stationary solution]\label{def:eps-sol} Let $\epsilon > 0$. A point $\X = [\vx_1, \ldots, \vx_n]$ is called an expected $\epsilon$-stationary solution of \eqref{eq:decentralized_problem} if for a constant $\lambda\in (0, \frac{1}{L})$,
    \vspace{-0.1cm}
\begin{equation*}
    \textstyle\frac{1}{n} \bbE\left[\sum_{i=1}^n \|\nabla \phi_\lambda(\vx_i)\|^2 + L^2 \|\X_\perp\|^2\right] \leq \epsilon^2.
    \vspace{-0.1cm}
\end{equation*}
\end{definition}  

In the definition above, $L^2$ before the consensus error term $\|\X_\perp\|^2$ is to balance the two terms. This scaling scheme has also been used in existing works such as \cite{xin2021stochastic,mancino2022proximal,DBLP:journals/corr/abs-2202-00255} . From the definition, we see that if $\X$ is an expected $\epsilon$-stationary solution of \eqref{eq:decentralized_problem}, then each local solution $\vx_i$ will be a near-stationary solution of $\phi$ and in addition, these local solutions are all close to each other, namely, they are near consensus.


Below we first state the convergence results of the non-compressed method DProxSGT 
and then the compressed one CDProxSGT. 
All the proofs are given in the appendix.
 
\begin{theorem}[Convergence rate of DProxSGT]\label{thm:sec2}
Under Assumptions \ref{assu:prob} -- \ref{assu:stoc_grad}, let $\{\X^t\}$ be generated from $\mathrm{DProxSGT}$ in Algorithm~\ref{alg:DProxSGT} with $\vx_i^0 = \vx^0, \forall\, i \in \hN$. Let  $\lambda = \min\big\{\frac{1}{4 L}, \frac{1}{96\rho L}\big\}$ and $\eta\leq  \min\big\{\frac{1}{4 L},\frac{(1-\rho^2)^4}{96\rho  L}\big\}$. 
Select $\tau$ from $\{0, 1, \ldots, T-1\}$ uniformly at random.
Then 
\vspace{-0.1cm}
\begin{equation*}
\begin{aligned} 
   &~ \textstyle \frac{1}{n} \bbE\left[\sum_{i=1}^n \|\nabla\phi_\lambda(\vx_i^\tau)\|^2 +\frac{4}{\lambda \eta} \|\X^\tau_\perp\|^2\right]  \\
   \leq &~ \textstyle \frac{8\left( \phi_\lambda(\vx^0) - \phi_\lambda^*\right)}{ \eta T}  + \frac{4616 \eta}{\lambda(1-\rho^2)^3} \sigma^2 \textstyle  + \frac{768\eta  \bbE\left[ \|\nabla \bF^0(\I-\J)\|^2\right]}{n\lambda T(1-\rho^2)^3},
\end{aligned}
\vspace{-0.1cm}
\end{equation*}
where $\phi_\lambda^* = \min_{\vx} \phi_\lambda(\vx)> -\infty$. 
\end{theorem}

By Theorem~\ref{thm:sec2}, we obtain a complexity result as follows.

 
\begin{corollary}[Iteration complexity]
Under the assumptions of Theorem~\ref{thm:sec2}, 
for a given $\epsilon>0$, take $ \eta = \min\{\frac{1}{4 L},\frac{(1-\rho^2)^4}{96\rho  L}, \frac{ \lambda(1-\rho^2)^3 \epsilon^2}{9232\sigma^2}\}$. Then $\mathrm{DProxSGT}$ can find an expected $\epsilon$-stationary point of \eqref{eq:decentralized_problem}  when $T \geq  T_\epsilon = \left\lceil \frac{16\left( \phi_\lambda(\vx^0) - \phi_\lambda^*\right)}{ \eta \epsilon^2 } + \frac{1536\eta  \bbE\left[ \|\nabla \bF^0(\I-\J)\|^2\right]}{n\lambda (1-\rho^2)^3 \epsilon^2} \right\rceil$.  
\end{corollary}
 
 
\begin{remark} \label{remark:DProxSGT}
When 
$\epsilon$ is small enough, $\eta$ will take $\frac{ \lambda(1-\rho^2)^3 \epsilon^2}{9232\sigma^2}$, and $T_\epsilon$ will be dominated by the first term. 
In this case, DProxSGT can find an expected $\epsilon$-stationary solution of \eqref{eq:decentralized_problem} in $O\Big( \frac{\sigma^2\left( \phi_\lambda(\vx^0) - \phi_\lambda^*\right) }{\lambda(1-\rho^2)^3 \epsilon^4}\Big)$ iterations, leading to the same number of stochastic gradient samples 
and communication rounds. Our sample complexity is optimal in terms of the dependence on $\epsilon$ under the smoothness condition in Assumption~\ref{assu:prob}, as it matches with the lower bound in \cite{arjevani2022lower}. However, the dependence on $1-\rho$ may not be optimal because of our possibly loose analysis, as the \emph{deterministic} method with single communication per update in \cite{scutari2019distributed} for nonconvex nonsmooth problems has a dependence $(1-\rho)^2$ on the graph topology. 
\end{remark}

 
\begin{theorem}[Convergence rate of CDProxSGT] \label{thm:sect3thm}
 Under Assumptions \ref{assu:prob} through \ref{assu:compressor},
 let $\{\X^t\}$ be generated from $\mathrm{CDProxSGT}$ in Algorithm \ref{alg:CDProxSGT} with
 $\vx_i^0 = \vx^0, \forall\, i \in \hN$. Let $\lambda = \min \big\{\frac{1}{4 L}, \frac{ (1-\alpha^2)^2}{9 L+41280}\big\}$, and suppose
\vspace{-0.1cm}
\begin{gather*}
 \eta \leq~ \min\left\{ \textstyle \lambda, 
 \frac{(1-\alpha^2)^2(1-\widehat\rho^2_x)^2(1-\widehat\rho^2_y)^2}{18830\max\{1, L\}}\right\}, \\
\gamma_x\leq~ \min\left\{ \textstyle 
\frac{1-\alpha^2}{25}, \frac{\eta}{\alpha}
\right\}, \quad 
 \gamma_y\leq ~   \textstyle 
 \frac{(1-\alpha^2)(1-\widehat\rho^2_x)(1-\widehat\rho^2_y)}{317}.
\end{gather*}
\vspace{-0.1cm}
Select $\tau$ from $\{0, 1, \ldots, T-1\}$ uniformly at random.
Then 
\vspace{-0.1cm}
\begin{equation*}
    \begin{aligned}
   &~ \textstyle \frac{1}{n} \bbE\left[\sum_{i=1}^n \|\nabla\phi_\lambda(\vx_i^\tau)\|^2 +\frac{4}{\lambda \eta} \|\X^\tau_\perp\|^2\right]   \\
   \leq &~\textstyle \frac{8\left(\phi_\lambda(\vx^0) - \phi_\lambda^*\right)}{\eta T}  
     +\frac{(50096n+48)\eta \sigma^2}{n\lambda(1-\widehat\rho^2_x)^2(1-\widehat\rho^2_y)} +   \frac{4176 \eta \bbE\left[ \|\nabla \bF^0\|^2\right] }{n\lambda T (1-\widehat\rho^2_x)^2(1-\widehat\rho^2_y)},
\end{aligned}
\vspace{-0.1cm}
\end{equation*}
where $\phi_\lambda^* = \min_{\vx} \phi_\lambda(\vx)> -\infty$. 
\end{theorem}

By Theorem~\ref{thm:sect3thm}, we have the complexity result as follows.

\begin{corollary}[Iteration complexity]
Under the assumptions of Theorem \ref{thm:sect3thm}, for a given $\epsilon>0$, take 
\begin{gather*}
    \eta = \textstyle \min \left\{\frac{1}{4 L}, \frac{ (1-\alpha^2)^2}{9 L+41280}, \frac{(1-\alpha^2)^2(1-\widehat\rho^2_x)^2(1-\widehat\rho^2_y)^2}{18830\max\{1, L\}}\right.,\\ \textstyle  \left. \frac{n\lambda(1-\widehat\rho^2_x)^2(1-\widehat\rho^2_y)\epsilon^2}{2(50096n+48) \sigma^2}\right\}, \\
   \textstyle  \gamma_x = \min\left\{ \textstyle  
\frac{1-\alpha^2}{25}, \frac{\eta}{\alpha}\right\}, \quad
 \gamma_y =  \frac{(1-\alpha^2)(1-\widehat\rho^2_x)(1-\widehat\rho^2_y)}{317}.
\end{gather*}
Then $\mathrm{CDProxSGT}$ can find an expected $\epsilon$-stationary point of \eqref{eq:decentralized_problem} when  $T\geq T_\epsilon^c$ where 
\begin{align*} 
    T_\epsilon^c = \textstyle \left\lceil\frac{16\left(\phi_\lambda(\vx^0) - \phi_\lambda^*\right)}{\eta \epsilon^2}  +   \frac{8352 \eta \bbE\left[ \|\nabla \bF^0\|^2\right] }{n\lambda (1-\widehat\rho^2_x)^2(1-\widehat\rho^2_y)\epsilon^2} \right\rceil.
\end{align*}

\end{corollary}

\begin{remark}
When the given tolerance $\epsilon$ is small enough,
$\eta$ will take $\frac{n\lambda(1-\widehat\rho^2_x)^2(1-\widehat\rho^2_y)\epsilon^2}{2(50096n+48) \sigma^2}$ and $T_\epsilon^c$ will be dominated by the first term. In this case, similar to DProxSGT in Remark \ref{remark:DProxSGT}, CDProxSGT can find an expected $\epsilon$-stationary solution of \eqref{eq:decentralized_problem} in $O\Big( \frac{\sigma^2\left( \phi_\lambda(\vx^0) - \phi_\lambda^*\right) }{ \lambda(1-\widehat\rho^2_x)^2(1-\widehat\rho^2_y) \epsilon^4}\Big)$ iterations. 
\end{remark} 

%% file: experiments.tex
\section{Numerical Experiments}\label{sec:numerical_experiments}
In this section, we test the proposed algorithms on training two neural network models, in order to demonstrate their better generalization over momentum variance-reduction methods and large-batch training methods and to demonstrate the success of handling heterogeneous data even when only compressed model parameter and gradient information are communicated among workers.
One neural network that we test is LeNet5  \cite{lecun1989backpropagation} on the FashionMNIST dataset \cite{xiao2017fashion}, and the other is FixupResNet20 \cite{zhang2019fixup} on Cifar10 \cite{krizhevsky2009learning}. 

Our experiments are representative to show the practical performance of our methods. Among several closely-related works,  \cite{xin2021stochastic} includes no experiments, and \cite{mancino2022proximal,zhao2022beer} only tests on tabular data and MNIST. \cite{koloskova2019decentralized-b} tests its method on Cifar10 but needs 
similar data distribution on all workers 
for good performance.
FashionMNIST has a similar scale as MNIST but poses a more challenging classification task \cite{xiao2017fashion}. 
Cifar10 is more complex, and FixupResNet20 has more layers than LeNet5.


 
All the compared algorithms are implemented in Python with Pytorch 
and MPI4PY (for distributed computing). 
They run on a Dell workstation with 
two Quadro RTX 5000 GPUs. We use the 2 GPUs as 5 workers, which communicate over a ring-structured network (so each worker can only communicate with two neighbors). Uniform weight is used, i.e., $W_{ji} = \frac{1}{3}$ for each pair of connected workers $i$ and $j$. 
Both FashionMNIST and Cifar10 have 10 classes. We distribute each data onto the 5 workers based on the class labels, namely, each worker holds 2 classes of data points, and thus the data are heterogeneous across the workers. 

For all methods, we report their objective values on training data, prediction accuracy on testing data, and consensus errors at each epoch. 
To save time, the objective values are computed as the average of the losses that are evaluated during the training process (i.e., on the sampled data instead of the whole training data) plus the regularizer per epoch. 
For the testing accuracy, we first compute the accuracy on the whole testing data for each worker by using its own model parameter and then take the average. 
The consensus error is simply $\|\X_\perp\|^2$. 

\subsection{Sparse Neural Network Training} \label{subsect:RegL1}
In this subsection, we test the non-compressed method DProxSGT and compare it with AllReduce (that is a centralized method and used as a baseline), DEEPSTORM\footnote{For DEEPSTORM, we implement DEEPSTORM v2 in \cite{mancino2022proximal}.} and ProxGT-SA \cite{xin2021stochastic} on solving \eqref{eq:decentralized_problem}, where  $f$ is the loss on the whole training data and $r(\vx) = \mu\|\vx\|_1$ serves as a sparse regularizer that encourages a sparse model. 

For training LeNet5 on FashionMNIST, we set $\mu= 10^{-4}$ and run each method to 100 epochs. The learning rate $\eta$ and batchsize are set to $0.01$ and 8 for AllReduce and DProxSGT. 
DEEPSTORM uses the same $\eta$ and batchsize but with a larger initial batchsize 200, 
and its momentum parameter is tuned to $\beta=0.8$ in order to yield the best performance. 
ProxGT-SA is a large-batch training method. 
We set its batchsize to 256 and accordingly apply a larger step size $\eta=0.3$ that is the best among $\{0.1, 0.2, 0.3, 0.4\}$. 

For training FixupResnet20  on Cifar10, we set $\mu= 5 \times 10^{-5}$ and run each method to 500 epochs. 
The learning rate and batchsize are set to $\eta=0.02$ and 64 for AllReduce, DProxSGT, and DEEPSTORM. The initial batchsize is set to 1600 for DEEPSTORM and the momentum parameter set to $\beta=0.8$. 
ProxGT-SA uses a larger batchsize 512 and a larger stepsize $\eta=0.1$ that gives the best performance among $\{0.05, 0.1, 0.2, 0.3\}$. 

\begin{figure}[ht] 
\begin{center}  
\includegraphics[width=.9\columnwidth]{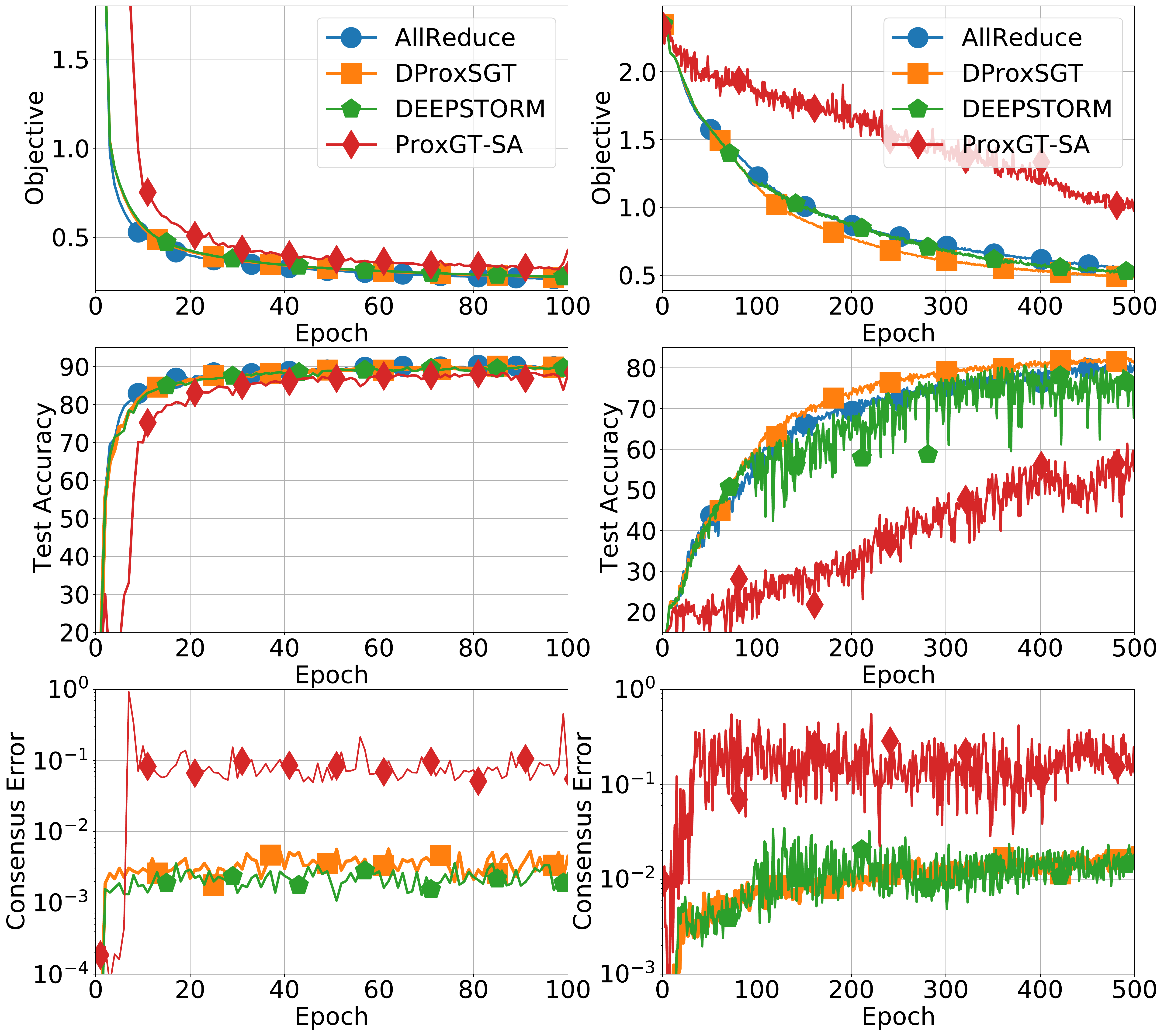} 
\vspace{-0.2cm}
\caption{Results of training sparse neural networks by non-compressed methods with $r(\vx) = \mu \|\vx\|_1$ for the same number of epochs. Left: LeNet5 on FashionMNIST with $\mu=10^{-4}$. Right: FixupResnet20 on Cifar10 with $\mu=5\times 10^{-5}$.}
\label{fig:RegL1}
\end{center} 
\end{figure}


The results for all methods are plotted in Figure \ref{fig:RegL1}. For LeNet5, DProxSGT produces almost the same curves as the centralized training method AllReduce, while on FixupResnet20, DProxSGT even outperforms AllReduce in terms of testing accuracy. This could be because AllReduce aggregates stochastic gradients from all the workers for each update and thus equivalently, it actually uses a larger batchsize. 
DEEPSTORM performs equally well as our method DProxSGT on training LeNet5. However, it gives lower testing accuracy than DProxSGT and also oscillates significantly more seriously on training the more complex neural network FixupResnet20. This appears to be caused by the momentum variance reduction scheme used in DEEPSTORM.
In addition, we see that the large-batch training method ProxGT-SA performs much worse than DProxSGT within the same number of epochs (i.e., data pass), especially on training FixupResnet20. 

\subsection{Neural Network Training by Compressed Methods} \label{subsect:compress}
In this subsection, we compare CDProxSGT with two state-of-the-art compressed training methods: Choco-SGD \cite{koloskova2019decentralized,koloskova2019decentralized-b} and BEER \cite{zhao2022beer}. As Choco-SGD and BEER are studied only for  problems without a regularizer, we set $r(\vx)=0$ in \eqref{eq:decentralized_problem} for the tests. Again, we compare their performance on training LeNet5 and FixupResnet20. 
The two non-compressed methods AllReduce and  DProxSGT are included as baselines.  
The same compressors are used for CDProxSGT, Choco-SGD, and BEER, when compression is applied. 

\begin{figure}[htbp] 
\begin{center} 
\includegraphics[width=.9\columnwidth]{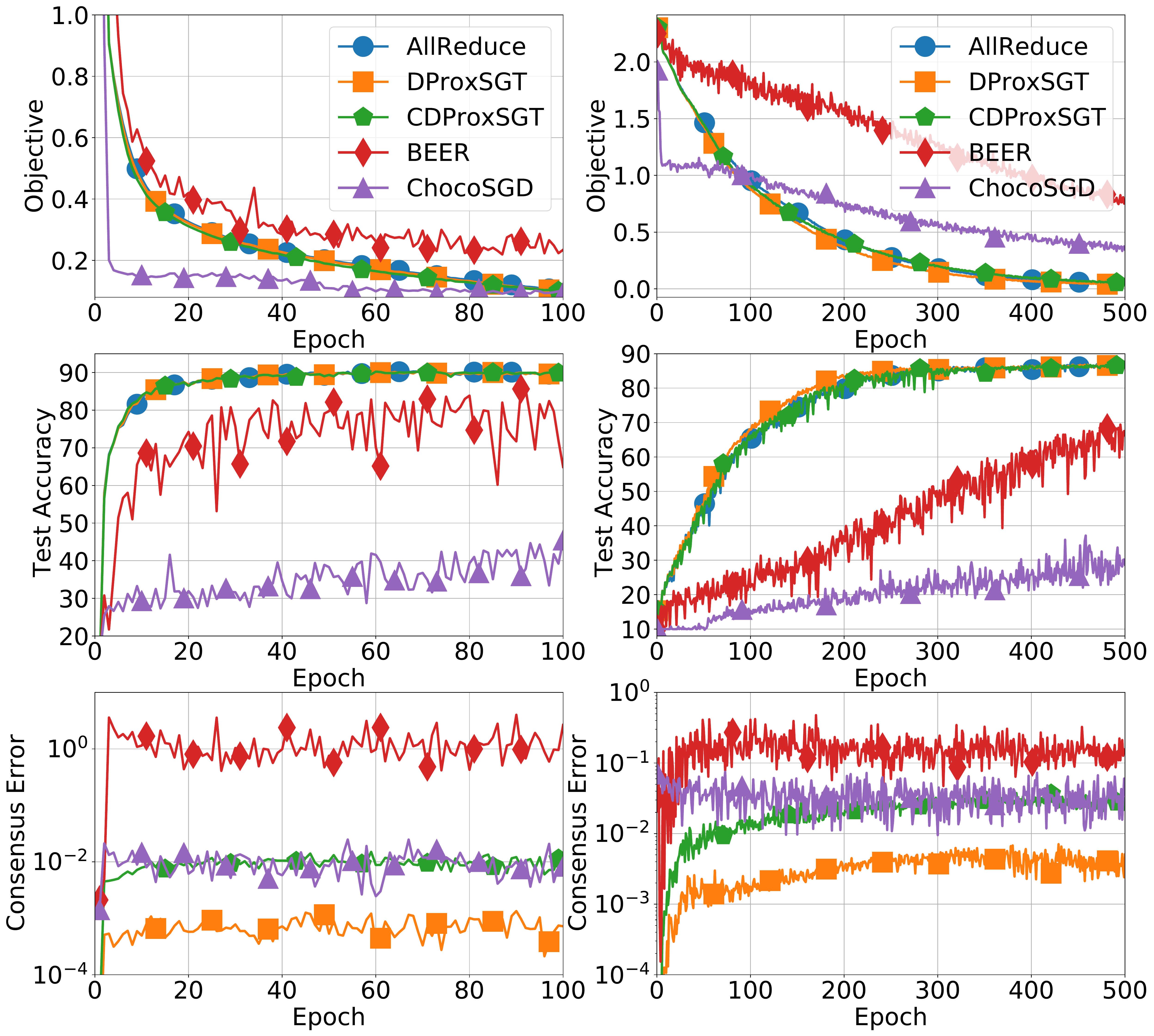}  
\vspace{-0.2cm}
\caption{Results of training neural network models by compressed methods for the same number of epochs.  Left: LeNet5 on FashionMNIST. Right: FixupResnet20 on Cifar10.}
\label{fig:Compress}
\end{center} 
\end{figure} 

We run each method to 100 epochs for training LeNet5 on FashionMNIST. 
The compressors $Q_y$  and $Q_x$ are set to top-$k(0.3)$ \cite{aji2017sparse}, i.e., taking the largest $30\%$ elements of an input vector in absolute values and zeroing out all others. 
We set batchsize to 8 and tune the learning rate $\eta$ to $0.01$ for AllReduce, DProxSGT, CDProxSGT and Choco-SGD, and for CDProxSGT, we set $\gamma_x=\gamma_y=0.5$. 
BEER is a large-batch training method. It uses a larger batchsize 256 and accordingly a larger learning rate $\eta=0.3$, which appears to be the best among $\{0.1, 0.2, 0.3, 0.4\}$. 
 
For training FixupResnet20 on the Cifar10 dataset, we run each method to 500 epochs. We take top-$k(0.4)$ \cite{aji2017sparse} as the compressors $Q_y$  and $Q_x$ and set $\gamma_x=\gamma_y=0.8$.
For AllReduce, DProxSGT, CDProxSGT and Choco-SGD, we set their batchsize to 64 and tune the learning rate $\eta$ to $0.02$. For BEER, we use a larger batchsize 512 and a larger learning rate $\eta=0.1$, which is the best among
$\{0.05, 0.1, 0.2, 0.3\}$. 

The results are shown in Figure \ref{fig:Compress}. 
For both models, CDProxSGT yields almost the same curves of objective values and testing accuracy as its non-compressed counterpart DProxSGT and the centralized non-compressed method AllReduce. This indicates about 70\% saving of communication for the training of LeNet5 and 60\% saving for FixupResnet20 without sacrifying the testing accuracy. 
In comparison, BEER performs significantly worse than the proposed method CDProxSGT within the same number of epochs in terms of all the three measures, especially on training the more complex neural network FixupResnet20, which should be attributed to the use of a larger batch by BEER. Choco-SGD can produce comparable objective values. However, its testing accuracy is much lower than that produced by our method CDProxSGT. 
This should be because of the data heterogeneity that ChocoSGD cannot handle, while CDProxSGT applies the gradient tracking to successfully address the challenges of data heterogeneity. 

%% file: appendix_proof_sect1.tex
\section{Some Key Existing Lemmas}

 
For $L$-smoothness function $f_i$, it holds for any $\vx, \vy\in\dom(r)$, 
\begin{align}\label{eq:assump-to-f_i}
\textstyle \big|f_i(\vy) - f_i(\vx) - \langle \nabla f_i(\vx), \vy-\vx\rangle\big| \le \frac{L}{2}\|\vy-\vx\|^2.
\end{align} 

From the smoothness of $f_i$ in Assumption \ref{assu:prob}, it follows that $f = \frac{1}{n}f_i$ is also $ L$-smooth in $\dom(r)$.

When $f_i$ is $ L$-smooth in $\dom(r)$, we have that $f_i(\cdot) + \frac{L}{2}\|\cdot\|^2$ is convex.  
Since $r(\cdot)$ is convex, $\phi_i(\cdot) + \frac{L}{2}\|\cdot\|^2$ is convex, i.e., $\phi_i$ is $L$-weakly convex for each $i$. So is $\phi$. In the following, we give some lemmas about weakly convex functions.

The following result is from Lemma II.1 in \cite{chen2021distributed}.
 \begin{lemma}\label{lem:weak_convx}
 For any function $\psi$ on $\bbR^{d}$, if it is $L$-weakly convex, i.e., $\psi(\cdot) + \frac{L}{2}\|\cdot\|^2$ is convex, then for any $\vx_1, \vx_2, \ldots, \vx_m\in\bbR^d$, it holds that 
\[
\psi\left(\sum_{i=1}^m a_i\vx_i\right)\leq \sum_{i=1}^m a_i \psi(\vx_i) + \frac{L}{2} \sum_{i=1}^{m-1} \sum_{j=i+1}^m a_i a_j \|\vx_i-\vx_j\|^2,
\]
where $a_i\geq 0$ for all $i$ and $\sum_{i=1}^m a_i=1$.
\end{lemma}

  The first result below is from Lemma II.8 in \cite{chen2021distributed}, and the nonexpansiveness of the proximal mapping of a closed convex function is well known. 
\begin{lemma} \label{lem:prox_diff} 
For any function $\psi$ on $\bbR^{d}$, if it is $L$-weakly convex, i.e., $\psi(\cdot) + \frac{L}{2}\|\cdot\|^2$ is convex, then the proximal mapping with $\lambda< \frac{1}{L}$ satisfies 
\[
\|\prox_{\lambda \psi}(\vx_1)-\prox_{\lambda \psi}(\vx_2)\|\leq \frac{1}{1-\lambda  L} \|\vx_1-\vx_2\|.
\]
For a closed convex function $r(\cdot)$,  its proximal mapping is nonexpansive, i.e., 
\[
\|\prox_{r}(\vx_1)-\prox_{r}(\vx_2)\|\leq  \|\vx_1-\vx_2\|.
\]
\end{lemma}

\begin{lemma}
For $\mathrm{DProxSGT}$ in Algorithm \ref{alg:DProxSGT} and $\mathrm{CDProxSGT}$ in Algorithm \ref{alg:CDProxSGT}, we both have
\begin{gather}
	\bar\vy^t =\overline{\nabla}  \bF^t, \quad
	\bar\vx^{t} = \bar\vx^{t+\frac{1}{2}} = \frac{1}{n} \sum_{i=1}^n  \prox_{\eta r}\left(\vx_i^t - \eta \vy_i^{t}\right). \label{eq:x_y_mean}
\end{gather}  
\end{lemma}
\begin{proof}
For DProxSGT in Algorithm \ref{alg:DProxSGT}, taking the average among the workers on \eqref{eq:y_half_update} to \eqref{eq:x_1_update} gives 
\begin{align}
\bar\vy^{t-\frac{1}{2}} = \bar\vy^{t-1} + \overline{\nabla}  \bF^t - \overline{\nabla}  \bF^{t-1}, \quad
    \bar\vy^t =\bar\vy^{t-\frac{1}{2}},  \quad
    \bar\vx^{t+\frac{1}{2}} = \frac{1}{n} \sum_{i=1}^n  \prox_{\eta r}\left(\vx_i^t - \eta \vy_i^{t}\right), \quad  \bar\vx^{t} = \bar\vx^{t+\frac{1}{2}},\label{eq:proof_mean}
\end{align}
where $\1^\top\W=\1^\top$ follows from Assumption \ref{assu:mix_matrix}. With $\bar\vy^{-1}=\overline{\nabla} \bF^{-1}$, we have \eqref{eq:x_y_mean}.

Similarly, for CDProxSGT in Algorithm \ref{alg:CDProxSGT}, 
taking the average on  \eqref{eq:alg3_1_matrix} to \eqref{eq:alg3_6_matrix} 
will also give \eqref{eq:proof_mean} and \eqref{eq:x_y_mean}.
\end{proof}

In the rest of the analysis, we define the Moreau envelope of $\phi$ for $\lambda\in(0,\frac{1}{L})$ as \begin{align*}
\phi_\lambda(\vx) = \min_\vy\left\{\phi(\vy) + \frac{1}{2\lambda}\|\vy-\vx\|^2\right\}.  
\end{align*}
Denote the minimizer as 
 \begin{align*}
 \prox_{\lambda \phi}(\vx):= \argmin_{\vy} \phi(\vy)+\frac{1}{2\lambda} \|\vy-\vx\|^2.
 \end{align*}
In addition, we will use the notation $\widehat{\vx}^t_i$ and $\widehat{\vx}^{t+\frac{1}{2}}_i$ that are defined by
\begin{align}
\widehat{\vx}^t_i = \prox_{\lambda \phi}(\vx^t_i),\ \widehat{\vx}^{t+\frac{1}{2}}_i = \prox_{\lambda \phi}(\vx^{t+\frac{1}{2}}_i),\, \forall\, i\in\hN,
\label{eq:x_t_hat} 
\end{align}
where $\lambda \in(0,\frac{1}{L})$.

%% file: appendix_proof_sect2.tex
\section{Convergence Analysis for DProxSGT} \label{sec:proof_DProxSGT}

In this section, we analyze the convergence rate of DProxSGT in Algorithm \ref{alg:DProxSGT}. For better readability, we use the matrix form of Algorithm \ref{alg:DProxSGT}. By the notation introduced in section~\ref{sec:notation}, we can write \eqref{eq:y_half_update}-\eqref{eq:x_1_update} in the more compact matrix form:
\begin{align}
 &   \Y^{t-\frac{1}{2}} =  \Y^{t-1}  +  \nabla \bF^t - \nabla \bF^{t-1},\label{eq:y_half_update_matrix}
    \\
 &   \Y^t =  \Y^{t-\frac{1}{2}}\W,\label{eq:y_update_matrix}\\
 &   \X^{t+\frac{1}{2}} =\prox_{\eta r} \left(\X^t - \eta \Y^{t}\right) \triangleq [\prox_{\eta r} \left(\vx_1^t - \eta \vy_1^{t}\right),\ldots,\prox_{\eta r} \left(\vx_n^t - \eta \vy_n^{t}\right)],\label{eq:x_half_update_matrix}
    \\
 &   \X^{t+1} = \X^{t+\frac{1}{2}}\W. \label{eq:x_1_update_matrix}
\end{align}   

Below, we first bound $\|\widehat\X^{t}-\X^{t+\frac{1}{2}}\|^2$ in Lemma~\ref{lem:Xhat_Xhalf}. Then we give the bounds of the consensus error $\|\X_\perp^t\|$ and $\|\Y_\perp^t\|$ and $\phi_\lambda(\vx_i^{t+1})$ after one step in Lemmas~\ref{lem:XI_J}, \ref{lem:YI_J}, and \ref{lem:weak_convex}. Finally, we prove Theorem \ref{thm:sec2} by constructing a Lyapunov function that involves $\|\X_\perp^t\|$, $\|\Y_\perp^t\|$, and $\phi_\lambda(\vx_i^{t+1})$.

 
\begin{lemma} \label{lem:Xhat_Xhalf} Let  $\eta\leq  \lambda \leq \frac{1}{4 L}$. Then
\begin{align}
 \bbE\big[\|\widehat\X^{t}-\X^{t+\frac{1}{2}}\|^2\big]  \leq  &~
4 \bbE\big[\|\X^t_\perp\|^2\big]   +  \left( 1-\frac{\eta}{2\lambda} \right) \bbE\big[\| \widehat\X^t - \X^t\|^2\big] +4\eta^2 \bbE\big[\|\Y^t_\perp\|^2\big] + 2\eta^2\sigma^2.  \label{eq:hatx_xprox}
\end{align} 
\end{lemma} 
\begin{proof}
By the definition of $\widehat\vx^t_i$ in \eqref{eq:x_t_hat}, we have $0  \in  \nabla f(\widehat\vx^t_i) + \partial r(\widehat\vx^t_i) + \frac{1}{\lambda}(\widehat\vx^t_i-\vx^t_i)$, i.e.,
\[ \textstyle
0 \in \partial r(\widehat\vx^t_i) + \frac{1}{\eta} \left(\frac{\eta}{\lambda} \widehat\vx^t_i-\frac{\eta}{\lambda}\vx^t_i + \eta  \nabla f(\widehat\vx^t_i)  \right) = \partial r(\widehat\vx^t_i) + \frac{1}{\eta} \left(\widehat\vx^t_i - \left( \frac{\eta}{\lambda}\vx^t_i - \eta  \nabla f(\widehat\vx^t_i)+ \left(1- \frac{\eta}{\lambda}\right) \widehat\vx^t_i  \right)\right).
\]
Thus we have $\widehat\vx^t_i = \prox_{\eta r}\left(  \frac{\eta}{\lambda}\vx^t_i - \eta  \nabla f(\widehat\vx^t_i) + \left(1- \frac{\eta}{\lambda}\right)\widehat\vx^t_i\right)$. Then by \eqref{eq:x_half_update}, the convexity of $r$, and Lemma \ref{lem:prox_diff}, 
\begin{align}
 &~  \textstyle \|\widehat\vx_i^{t}-\vx_i^{t+\frac{1}{2}}\|^2
 = \left\| \prox_{\eta r}\left( \frac{\eta}{\lambda}\vx^t_i - \eta  \nabla f(\widehat\vx^t_i) + \left(1- \frac{\eta}{\lambda}\right)\widehat\vx^t_i\right)- \prox_{\eta r} \left(
 \vx_i^t - \eta \vy^t_i\right) \right\|^2 \nonumber\\
 \leq &~    \textstyle \left\| \frac{\eta}{\lambda}\vx^t_i - \eta  \nabla f(\widehat\vx^t_i) + \left(1- \frac{\eta}{\lambda}\right)\widehat\vx^t_i - (\vx^t_i-\eta\vy^t_i) \right\|^2  = \left\| \left(1- \frac{\eta}{\lambda}\right)(\widehat\vx^t_i -\vx^t_i )- \eta  (\nabla f(\widehat\vx^t_i) -\vy^t_i) \right\|^2 \nonumber\\
 = & ~ \textstyle  \left(1- \frac{\eta}{\lambda}\right)^2  \left\| \widehat\vx^t_i - \vx^t_i \right\|^2 + \eta^2\left\| \vy^t_i- \nabla f(\widehat\vx^t_i) \right\|^2 + 2  \left(1- \frac{\eta}{\lambda}\right)\eta \left\langle \widehat\vx^t_i-\vx_i^t, \vy_i^t-\nabla f(\vx^t_i) + \nabla f(\vx^t_i)-\nabla f(\widehat\vx^t_i) \right\rangle\nonumber\\
\leq & ~ \textstyle\left(\left(1- \frac{\eta}{\lambda}\right)^2 + 2\left(1- \frac{\eta}{\lambda}\right)\eta  L
\right) \left\| \widehat\vx^t_i - \vx^t_i \right\|^2 + \eta^2\left\| \vy^t_i- \nabla f(\widehat\vx^t_i) \right\|^2 + 2  \left(1- \frac{\eta}{\lambda}\right)\eta \left\langle \widehat\vx^t_i-\vx_i^t, \vy_i^t-\nabla f(\vx^t_i) \right\rangle , \label{eq:lem1.6.1} 
 \end{align}
where 
the second inequality holds by $\left\langle \widehat\vx^t_i-\vx_i^t, \nabla f(\vx^t_i)-\nabla f(\widehat\vx^t_i) \right\rangle \leq  L\left\|\widehat\vx^t_i-\vx_i^t\right\|^2$. 
The second term in the right hand side of \eqref{eq:lem1.6.1} can be bounded by
\begin{align*}
&~  \textstyle \bbE_t [\| \vy^t_i- \nabla f(\widehat\vx^t_i) \|^2\big]  \overset{\eqref{eq:x_y_mean}}{=} \bbE_t\big[\| \vy^t_i- \bar\vy^t + \overline{\nabla} \bF^t  - \nabla f(\widehat\vx^t_i) \|^2\big]  
\leq   2\bbE_t\big[\| \vy^t_i- \bar\vy^t \|^2\big] + 2\bbE_t\big[\big\| \overline{\nabla} \bF^t - \nabla f(\widehat\vx^t_i) \big\|^2\big] \\
= &~2\bbE_t\big[\| \vy^t_i- \bar\vy^t \|^2\big] + 2\bbE_t\big[\| \overline{\nabla} \bF^t - \overline{\nabla} \f^t \|^2\big]+ 2\| \overline{\nabla} \f^t - \nabla f(\widehat\vx^t_i) \|^2  \\
\leq&~ 2\bbE_t[ \|\vy_i^t-\bar\vy^t\|^2\big] +  \frac{2}{n^2}\sum_{j=1}^n \bbE_t\big[\|\nabla F_j(\vx_j^t,\xi_j^t)-\nabla f_j(\vx_j^t)\|^2\big]  + 4 \| \overline{\nabla} \f^t -\nabla f(\vx^t_i) \|^2 + 4\| \nabla f(\vx^t_i)- \nabla f(\widehat\vx^t_i) \|^2\\
\leq &~ 2\bbE_t[ \|\vy_i^t-\bar\vy^t\|^2\big] + 2 \frac{\sigma^2}{n} + 4 \| \overline{\nabla} \f^t -\nabla f(\vx^t_i) \|^2  + 4  L^2 \|\vx^t_i -\widehat\vx^t_i\|^2,
\end{align*}
where the second equality holds by the unbiasedness of stochastic gradients, and the second inequality holds also by the independence between $\xi_i^t$'s. 
In the last inequality, we use the bound of the variance of stochastic gradients, and the $L$-smooth assumption.
Taking the full expectation over the above inequality and summing for all $i$ give
\begin{align}
\sum_{i=1}^n\bbE\big[\| \vy^t_i- \nabla f(\widehat\vx^t_i) \|^2 ] 
\leq 2\bbE\big[\|\Y^t_\perp\|^2 ] +2\sigma^2 +  8 L^2  \bbE\big[\|\X^t_\perp\|^2 ]  +4 L^2 \bbE\big[\| \X^t -  \widehat\X^t\|^2 ]. \label{eq:lem161_1}
\end{align}
To have the inequality above, we have used
\begin{align}
&~ \sum_{i=1}^n \left\| \overline{\nabla} \f^t -\nabla f(\vx^t_i) \right\|^2  
\leq  \frac{1}{n} \sum_{i=1}^n \sum_{j=1}^n  \left\|\nabla f_j(\vx_j^t) -\nabla f_j(\vx^t_i) \right\|^2 
\leq   \frac{ L^2}{n}\sum_{i=1}^n\sum_{j=1}^n  
\left\|\vx_j^t - \vx^t_i \right\|^2  \nonumber \\
= &~ \frac{ L^2}{n}\sum_{i=1}^n\sum_{j=1}^n \left(  \left\|\vx_j^t - \bar\vx^t
\right\|^2 +\left\|\bar\vx^t-\vx^t_i \right\|^2  + 2\left\langle \vx_j^t - \bar\vx^t, \bar\vx^t-\vx^t_i\right\rangle\right) 
= 
2 L^2 \left\|\X^t_\perp\right\|^2, \label{eq:sumsum}
\end{align}
where the last equality holds by $ \frac{1}{n} \sum_{i=1}^n\sum_{j=1}^n \left\langle \vx_j^t - \bar\vx^t, \bar\vx^t-\vx^t_i\right\rangle =  \sum_{i=1}^n \left\langle \frac{1}{n} \sum_{j=1}^n (\vx_j^t - \bar\vx^t), \bar\vx^t-\vx^t_i\right\rangle =\sum_{i=1}^n\left\langle   \bar\vx^t - \bar\vx^t, \bar\vx^t-\vx^t_i\right\rangle=0$ from the definition of $\bar\vx$. 

About the third term in the right hand side of \eqref{eq:lem1.6.1}, we have
\begin{align}
 & ~\sum_{i=1}^n \bbE\left[ \left\langle \widehat\vx^t_i-\vx_i^t, \vy_i^t-\nabla f(\vx^t_i) \right\rangle\right] \overset{\eqref{eq:x_y_mean}}{=} \sum_{i=1}^n \bbE\left[\left\langle \widehat\vx^t_i-\vx_i^t, \vy_i^t -\bar\vy^t+\overline{\nabla} \bF^t -\nabla f(\vx^t_i) \right\rangle\right] \nonumber \\
= & ~  \textstyle \sum_{i=1}^n  \bbE\big[ \langle\widehat\vx^t_i -\bar{\widehat\vx}^t, \vy_i^t -\bar\vy^t \rangle\big] +  \sum_{i=1}^n    \bbE\big[\langle \bar{\vx}^t - \vx_i^t,\vy_i^t -\bar\vy^t \rangle\big]   +   \sum_{i=1}^n \bbE\left[ \left\langle \widehat\vx^t_i-\vx_i^t, \bbE_{t} \left[\overline{\nabla} \bF^t\right] -\nabla f(\vx^t_i)\right\rangle\right] \nonumber \\
 \leq &~ \frac{1}{2\eta} \left( \textstyle  \bbE\big[\|\widehat\X^t_\perp\|^2\big]+ \bbE\big[\|\X^t_\perp\|^2\big]\right) + \eta \bbE\big[\|\Y^t_\perp\|^2\big]   +  \textstyle  L\bbE\big[\| \widehat\X^t-\X^t\|^2\big] +  \frac{1}{4 L} \sum_{i=1}^n \bbE\big[\|\overline{\nabla} \vf^t -\nabla f(\vx^t_i)\|^2\big] 
 \nonumber \\
\leq &~ \left(\textstyle\frac{1}{2\eta(1-\lambda L)^2} + \frac{1}{2\eta}  + \frac{L}{2}\right)\bbE\big[\|\X^t_\perp\|^2\big] + \eta\bbE\big[\|\Y^t_\perp\|^2\big] +  L\bbE\big[\|\widehat\X^t-\X^t\|^2\big],\label{eq:lem161_2}
\end{align}
where $\textstyle \sum_{i=1}^n  \big\langle \bar{\widehat\vx}^t,\vy_i^t -\bar\vy^t \big\rangle = 0$ and  $\sum_{i=1}^n  \left\langle \bar{\vx}^t,\vy_i^t -\bar\vy^t \right\rangle = 0$ is used in the second equality, $\bbE_{t} \left[\overline{\nabla} \bF^t\right] = \overline{\nabla} \vf^t$ is used in the first inequality, and  $\|\widehat\X^t_\perp\|^2 =\left\|\left(\prox_{\lambda \phi}(\X^t)- \prox_{\lambda \phi}(\bar\vx^t)\1^\top\right) (\I-\J)\right\|^2\leq \frac{1}{(1-\lambda  L)^2}\|\X^t-\bar\X^t\|^2$ and \eqref{eq:sumsum} are used in the last inequality.

Now we can bound the summation of \eqref{eq:lem1.6.1} by using  \eqref{eq:lem161_1} and \eqref{eq:lem161_2}:
\begin{align*}
& ~ \bbE\big[\|\widehat\X^{t}-\X^{t+\frac{1}{2}}\|^2\big]\\
  \leq & ~ \left(\textstyle \left(1- \frac{\eta}{\lambda}\right)^2 + 2\left(1- \frac{\eta}{\lambda}\right)\eta L
\right) \bbE\big[\| \widehat\X^t - \X^t \|^2\big] \\
& ~ + \eta^2 \left(2\bbE[ \|\Y^t_\perp\|^2\big] +2\sigma^2 +  8 L^2  \bbE\big[\|\X^t_\perp\|^2\big] +4 L^2  \bbE\big[\| \X^t -  \widehat\X^t\|^2\big]\right) \\
& ~ + \textstyle 2  \left(1- \frac{\eta}{\lambda}\right)\eta  \left(\textstyle\left(\frac{1}{2\eta(1-\lambda L)^2} + \frac{1}{2\eta}  + \frac{ L}{2} \right)\bbE\big[\|\X^t_\perp\|^2\big] + \eta\bbE\big[\|\Y^t_\perp\|^2\big] +  L\bbE\big[\|\widehat\X^t-\X^t\|^2\big]\right) \\ 
= & ~ \textstyle \left(1 - 2\eta (\frac{1}{\lambda} - 2 L) + \frac{\eta^2}{\lambda} (\frac{1}{\lambda} - 2 L) + 2 L \eta^2(-\frac{1}{\lambda}+2 L)\right) \bbE\big[\| \widehat\X^t - \X^t\|^2\big] + 2\eta^2\sigma^2   \nonumber\\
 & ~ +  \textstyle \left( \left(1- \frac{\eta}{\lambda}\right) (1+\frac{1}{(1-\lambda L)^2}+ \eta  L) + 8\eta^2 L^2 \right) \bbE\big[\|\X^t_\perp\|^2\big]
  + 2 (2- \frac{\eta}{\lambda}) \eta^2 \bbE\big[\|\Y^t_\perp\|^2\big]. 
\end{align*} 
With  $\eta \leq   \lambda \leq \frac{1}{4 L}$, we have $\frac{1}{(1-\lambda L)^2}\leq 2$ and 
\eqref{eq:hatx_xprox} follows from the inequality above.
 
\end{proof}
 
\begin{lemma}\label{lem:XI_J} 
The consensus error of $\X$ satisfies the following inequality
\begin{align} 
 \bbE\big[\|\X^t_\perp\|^2\big] \leq  \frac{1+\rho^2}{2} \bbE\big[\|   \X^{t-1}_\perp \|^2\big]+ \frac{2\rho^2  \eta^2 }{1-\rho^2}  \bbE\big[\| \Y^{t-1}_\perp \|^2\big]. \label{eq:X_consensus}
\end{align}
\end{lemma}
\begin{proof}
 With the updates \eqref{eq:x_half_update} and \eqref{eq:x_1_update}, we have
\begin{align*}
     &~ 
      \bbE\big[\|\X^t_\perp\|^2\big] = \bbE\big[\|\X^{t-\frac{1}{2}}\W(\I- \J)\|^2\big] = \bbE\big[\|\X^{t-\frac{1}{2}} (\W-\J)\|^2\big] \nonumber\\
    =&~ \bbE\big[\| \prox_{\eta r} \left(\X^{t-1} - \eta \Y^{t-1}\right) (\W-\J)\|^2\big] \nonumber\\
    =&~ \bbE\big[\| \left(\prox_{\eta r} \left(\X^{t-1} - \eta \Y^{t-1}\right)-\prox_{\eta r} \left(\bar\vx^{t-1} - \eta \bar\vy^{t-1}\right)\1^\top\right) (\W-\J)\|^2\big] \nonumber\\
    \leq &~ \bbE\big[\|\prox_{\eta r} \left(\X^{t-1} - \eta \Y^{t-1}\right)-\prox_{\eta r} \left(\bar\vx^{t-1} - \eta \bar\vy^{t-1}\right)\1^\top\|^2 \|(\W-\J)\|^2_2] \nonumber\\
    \leq &~  \rho^2 \bbE\left[ \textstyle \sum_{i=1}^n\| \prox_{\eta r} \left(\vx_i^{t-1} - \eta \vy_i^{t-1}\right)-\prox_{\eta r} \left(\bar\vx^{t-1} - \eta \bar\vy^{t-1}\right)  \|^2\right] \nonumber\\
    \leq &~ \rho^2  \bbE\left[ \textstyle \sum_{i=1}^n\|\left(\vx_i^{t-1} - \eta \vy_i^{t-1}\right)-\left(\bar\vx^{t-1} - \eta \bar\vy^{t-1}\right)  \|^2\right]  = \rho^2  \bbE\big[\|   \X^{t-1}_\perp - \eta    \Y^{t-1}_\perp \|^2\big] \nonumber\\
  \leq &~ \textstyle \big(\textstyle \rho^2 + \frac{1-\rho^2}{2}\big) \bbE\big[\|   \X^{t-1}_\perp \|^2\big]+ \big( \textstyle\rho^2 + \frac{2\rho^4}{1-\rho^2}\big) \eta^2\bbE\big[\| \Y^{t-1}_\perp \|^2\big] \nonumber\\
  = &~\textstyle \frac{1+\rho^2}{2} \bbE\big[\|   \X^{t-1}_\perp \|^2\big]+ \frac{1+\rho^2}{1-\rho^2}  \rho^2 \eta^2\bbE\big[\| \Y^{t-1}_\perp \|^2\big] \nonumber \\
  \leq &~\textstyle \frac{1+\rho^2}{2} \bbE\big[\|   \X^{t-1}_\perp \|^2\big]+ \frac{2\rho^2  \eta^2 }{1-\rho^2}  \bbE\big[\| \Y^{t-1}_\perp \|^2\big],
\end{align*}  
where we have used $\1^\top (\W-\J)=\0$ in the third equality, $\|\W-\J\|_2\leq \rho$ in the second inequality, and Lemma \ref{lem:prox_diff} in the third inequality, and $\rho\leq 1$ is used in the last inequality.
\end{proof}
 
\begin{lemma}\label{lem:YI_J}
Let $\eta\leq \min\{\lambda, \frac{1-\rho^2}{4\sqrt{6} \rho  L} \} $ and $\lambda \leq\frac{1}{4 L}$.  The consensus error of $\Y$ satisfies
\begin{align} 
   \bbE\big[\|\Y^t_\perp\|^2\big]  
   \leq &~ \frac{48\rho^2  L^2 }{1-\rho^2 }   \bbE\big[\|\X^{t-1}_\perp\|^2\big] \!+\! \frac{3\!+\!\rho^2}{4}  \bbE\big[\|\Y^{t-1}_\perp \|^2\big] \!+\! \frac{12\rho^2  L^2 }{1-\rho^2 } \bbE\big[\|\widehat\X^{t-1}-\X^{t-1} \|^2\big] \!+\! 6 n\sigma^2. \label{eq:Y_consensus}
\end{align}
\end{lemma}

\begin{proof}
By the updates \eqref{eq:y_half_update} and \eqref{eq:y_update}, we have  
\begin{align}	 
     &~ 
      \bbE\big[\|\Y^t_\perp\|^2\big] = \bbE\big[\|\Y^{t-\frac{1}{2}}(\W- \J)\|^2\big] 
  =  \bbE\big[\| \Y^{t-1}(\W -\J)  +  (\nabla \bF^t - \nabla \bF^{t-1}) (\W -\J)\|^2\big] \nonumber\\
  = &~ \bbE\big[\|\Y^{t-1}(\I-\J)(\W -\J)\|^2\big] +  \bbE\big[\|(\nabla \bF^t - \nabla \bF^{t-1})  (\W-\J) \|^2\big] + 2\bbE\big[\langle \Y^{t-1} (\W -\J), (\nabla \bF^t - \nabla \bF^{t-1})  (\W-\J) \rangle\big] \nonumber\\
   \leq &~   \rho^2 \bbE\big[\|\Y^{t-1}_\perp \|^2\big] +  \rho^2 \bbE\big[\|\nabla \bF^t - \nabla \bF^{t-1}\|^2\big] +  2\bbE\big[\langle \Y^{t-1} (\W -\J),(\nabla \f^t - \nabla \bF^{t-1})(\W-\J) \rangle\big], \label{eq:y_cons1}
\end{align} 
where we have used $\J\W=\J\J=\J$, $\|\W-\J\|_2\leq \rho$ 
and $\bbE_t[\nabla \bF^t] = \nabla \vf^t$. 
For the second term on the right hand side of \eqref{eq:y_cons1}, we have
\begin{align}
   &~\bbE\big[\|\nabla \bF^t - \nabla \bF^{t-1}\|^2\big] = \bbE\big[\|\nabla \bF^t - \nabla \f^t+\nabla \f^t -\nabla \bF^{t-1}\|^2\big] \nonumber\\
 \overset{\bbE_t[\nabla \bF^t] = \nabla \f^t}{=}&~
  \bbE\big[\|\nabla \bF^t - \nabla \f^t\|^2\big]+\bbE\big[\|\nabla \f^t- \nabla \f^{t-1}+\nabla \f^{t-1}-\nabla \bF^{t-1}\|^2\big] \nonumber \\ 
\leq &~  \bbE\big[\|\nabla \bF^t - \nabla \f^t\|^2\big]+2\bbE\big[\|\nabla \f^t- \nabla \f^{t-1}\|^2\big]+2\bbE\big[\|\nabla \f^{t-1}-\nabla \bF^{t-1}\|^2\big]  \nonumber\\
 \leq &~  3 n \sigma^2 + 2 L^2 \bbE\big[\|\X^{t}-\X^{t-1}\|^2\big]. \label{eq:y_cons12}
\end{align}

For the third term on the right hand side of \eqref{eq:y_cons1}, we have
\begin{align}
&~2\bbE\big[\langle \Y^{t-1} (\W -\J), (\nabla \f^t - \nabla \bF^{t-1})(\W-\J) \rangle\big] \nonumber \\ 
 =&~2\bbE\big[\langle \Y^{t-1}(\W -\J), (\nabla \f^t - \nabla \f^{t-1})(\W-\J) \rangle\big] +2\bbE\big[\langle  \Y^{t-1}(\W -\J), (\nabla \f^{t-1} - \nabla \bF^{t-1})(\W-\J) \rangle\big] \nonumber \\
 =&~2\bbE\big[\langle \Y^{t-1}(\I-\J)(\W -\J), (\nabla \f^t - \nabla \f^{t-1})(\W-\J) \rangle\big] \nonumber \\
 &~ +2\bbE\big[\langle  (\Y^{t-2}  +  \nabla \bF^{t-1} - \nabla \bF^{t-2})\W(\W -\J), (\nabla \f^{t-1} - \nabla \bF^{t-1})(\W-\J) \rangle\big] \nonumber \\
 =&~2\bbE\big[\langle \Y^{t-1}(\I-\J)(\W -\J), (\nabla \f^t - \nabla \f^{t-1})(\W-\J) \rangle\big] \nonumber \\
 &~ +2\bbE\big[\langle  (\nabla \bF^{t-1} - \nabla \f^{t-1} )\W(\W -\J), (\nabla \f^{t-1} - \nabla \bF^{t-1})(\W-\J) \rangle\big] \nonumber \\
 \leq &~2\bbE\big[\|\Y^{t-1}(\I-\J)(\W -\J)\|\cdot\|(\nabla \f^t - \nabla \f^{t-1})(\W-\J) \|\big] \nonumber \\
 &~ +2\bbE\big[\|(\nabla \bF^{t-1} - \nabla \f^{t-1} )\W(\W -\J)\|\cdot\|(\nabla \f^{t-1} - \nabla \bF^{t-1})(\W-\J)\|\big] \nonumber \\
 \leq&~ 2\rho^2\bbE\big[\| \Y^{t-1}_\perp\|\cdot\|\nabla \f^t - \nabla \f^{t-1}\|\big]  +  2\rho^2\bbE\big[\|\nabla \bF^{t-1} - \nabla \f^{t-1}\|^2\big] \nonumber\\
 \leq &~ \textstyle\frac{1-\rho^2}{2} \bbE\big[\| \Y^{t-1}_\perp\|^2\big]+\frac{2\rho^4}{1-\rho^2}\bbE\big[\|\nabla \f^t - \nabla \f^{t-1}\|^2\big] + 2\rho^2 n \sigma^2 \nonumber \\
\leq &~  \textstyle\frac{1-\rho^2}{2} \bbE\big[\| \Y^{t-1}_\perp\|^2\big]+\frac{2\rho^4 L^2}{1-\rho^2} \bbE\big[\|  \X^t - \X^{t-1}\|^2\big]+ 2\rho^2 n \sigma^2, \label{eq:y_cons13}
\end{align}
where the second equality  holds by $\W-\J=(\I-\J)(\W-\J)$, \eqref{eq:y_half_update} and \eqref{eq:y_update}, the third equality holds because $\Y^{t-2} - \nabla \bF^{t-2} -\nabla \f^{t-1}$ does not depend on $\xi_i^{t-1}$'s, 
and the second inequality holds because $\|\W-\J\|_2\leq \rho$ and $\|\W\|_2\leq 1$.  
Plugging  \eqref{eq:y_cons12} and \eqref{eq:y_cons13} into \eqref{eq:y_cons1}, we have 
\begin{align} 
    \bbE\big[\|\Y^t_\perp\|^2\big] 
   \leq &~ \textstyle\frac{1+\rho^2}{2} \bbE\big[\|\Y^{t-1}_\perp \|^2\big]  + \frac{2 \rho^2 L^2 }{1-\rho^2 } \bbE\big[\| \X^t - \X^{t-1}\|^2\big]  +  5 \rho^2 n \sigma^2 , \label{eq:y_cons2} 
   \end{align}
where we have used $1+\frac{\rho^2}{1-\rho^2} = \frac{1}{1-\rho^2 }$.
For the second term in the right hand side of \eqref{eq:y_cons2}, we have
\begin{align}
&~ \| \X^{t+1} - \X^{t}\|^2 = \|\X^{t+\frac{1}{2}}\W-\X^t\|^2 =
\|(\X^{t+\frac{1}{2}}-\widehat\X^t)\W +(\widehat\X^t-\X^t)\W + \X^t (\W-\I)\|^2 \nonumber \\
\leq &~   3\|(\X^{t+\frac{1}{2}}-\widehat\X^t)\W\|^2 +3\|(\widehat\X^t-\X^t)\W\|^2 + 3\|\X^t(\I-\J)(\W-\I)\|^2 \nonumber \\
\leq &~  3\|\X^{t+\frac{1}{2}}-\widehat\X^t \|^2 +3\|\widehat\X^t-\X^t \|^2 + 12\|\X^t_\perp\|^2,\label{eq:Xplus1-X}
\end{align} 
where in the first inequality we have used $\X^t (\W-\I)=\X^t(\I-\J)(\W-\I)$ from $\J(\W-\I) = \J-\J$, and in the second inequality we have used $\|\W\|_2\leq 1$ and $\|\W-\I\|_2\leq 2$.

Taking expectation over both sides of \eqref{eq:Xplus1-X} and using \eqref{eq:hatx_xprox}, we have
\begin{align*}
&~ \bbE\big[\| \X^{t+1} - \X^{t}\|^2\big] \\
\le &~3 \left( \textstyle 4 \bbE\big[\|\X^t_\perp\|^2\big]   +  \left( 1-\frac{\eta}{2\lambda} \right) \bbE\big[\| \widehat\X^t - \X^t\|^2\big] +4\eta^2 \bbE\big[\|\Y^t_\perp\|^2\big] + 2\eta^2\sigma^2\right)  +3 \bbE\big[\|\widehat\X^t-\X^t \|^2\big] + 12 \bbE\big[\|\X^t_\perp\|^2\big]\\
= &~ 3 \textstyle  \left(2 -\frac{\eta}{2\lambda} \right) \bbE\big[\| \widehat\X^t - \X^t\|^2\big] +12\eta^2 \bbE\big[\|\Y^t_\perp\|^2\big] + 6\eta^2\sigma^2 + 24\bbE\big[\|\X^t_\perp\|^2\big].
\end{align*}
Plugging the inequality above into   \eqref{eq:y_cons2} gives  
\begin{align*} 
  \bbE\big[\|\Y^t_\perp\|^2\big] 
   \leq &~ \left(\textstyle\frac{1+\rho^2}{2} + \frac{24 \rho^2  L^2\eta^2 }{1-\rho^2 } \right) \bbE\big[\|\Y^{t-1}_\perp \|^2\big]  +  \textstyle 5 \rho^2  n\sigma^2 +\frac{12 \rho^2 L^2  \eta^2 \sigma^2 }{1-\rho^2 } \nonumber \\
   &~\textstyle + \frac{6\rho^2  L^2 }{1-\rho^2 }\left( \textstyle 2- \frac{\eta}{2\lambda}  \right) \bbE\big[\|\widehat\X^{t-1}-\X^{t-1} \|^2\big] + \frac{48 \rho^2 L^2 }{1-\rho^2 }   \bbE\big[\|\X^{t-1}_\perp\|^2\big].
\end{align*}
By $\rho<1$ and  $ \eta \leq \frac{1-\rho^2}{4\sqrt{6} \rho  L}$,  we have 
$\frac{24 \rho^2  L^2 \eta^2}{1-\rho^2 } \leq \frac{1-\rho^2}{4}$ and  $\frac{12 \rho^2  L^2  \eta^2}{1-\rho^2 } \leq \frac{1-\rho^2}{8}\leq n$, and further \eqref{eq:Y_consensus}.
\end{proof}

\begin{lemma}\label{lem:weak_convex} 
Let  $\eta\leq \lambda \leq\frac{1}{4 L}$. It holds
\begin{align}
 \sum_{i=1}^n  \bbE[\phi_\lambda(\vx_i^{t+1})]
\leq  &~   \sum_{i=1}^n \bbE[ \phi_\lambda( \vx_i^{t})] + \frac{4}{\lambda}
  \bbE\big[\|\X^t_\perp\|^2\big] +   \frac{4 \eta^2}{\lambda} \bbE[ \|\Y^t_\perp\|^2\big]   - \frac{\eta}{4\lambda^2} \bbE\big[\| \widehat\X^t - \X^t \|^2\big]  + \frac{\eta^2\sigma^2}{\lambda}. \label{eq:phi_update}
\end{align}
\end{lemma}
\begin{proof}
By the definition in \eqref{eq:x_t_hat}
, the update in \eqref{eq:x_1_update}, the $ L$-weakly convexity of $\phi$, and the convexity of $\|\cdot\|^2$, we have
\begin{align}
&~\phi_\lambda(\vx_i^{t+1}) \overset{\eqref{eq:x_t_hat}}{=}  \phi(\widehat\vx_i^{t+1})+{\textstyle \frac{1}{2\lambda} }\|\widehat\vx_i^{t+1}-\vx_i^{t+1}\|^2  \overset{\eqref{eq:x_1_update}}{\leq}  \phi\bigg(\sum_{j=1}^n\W_{ji}\widehat\vx_j^{t+\frac{1}{2}}\bigg)+{ \frac{1}{2\lambda}} \bigg\|\sum_{j=1}^n \W_{ji}\big(\widehat\vx_j^{t+\frac{1}{2}}-\vx_j^{t+\frac{1}{2}}\big)\bigg\|^2  \nonumber \\
&~\overset{\mbox{Lemma \ref{lem:weak_convx}} }{\leq} \sum_{j=1}^n \W_{ji} \phi(\widehat\vx_j^{t+\frac{1}{2}}) +{ \frac{L}{2} }\sum_{j=1}^{n-1}\sum_{l=j+1}^n \W_{ji}\W_{li}\|\widehat\vx_j^{t+\frac{1}{2}}-\widehat\vx_l^{t+\frac{1}{2}}\|^2+{  \frac{1}{2\lambda} }\sum_{j=1}^n \W_{ji} \|\widehat\vx_j^{t+\frac{1}{2}}-\vx_j^{t+\frac{1}{2}}\|^2 \nonumber \\
&~\leq \sum_{j=1}^n \W_{ji} \phi_\lambda(\vx_j^{t+\frac{1}{2}}) + \frac{1}{4\lambda} \sum_{j=1}^{n-1}\sum_{l=j+1}^n \W_{ji}\W_{li}\|\vx_j^{t+\frac{1}{2}}-\vx_l^{t+\frac{1}{2}}\|^2,  \label{eq:phi_update_0}
\end{align}
where 
in the last inequality we use $  \phi(\widehat\vx_j^{t+\frac{1}{2}}) + \frac{1}{2\lambda} \|(\widehat\vx_j^{t+\frac{1}{2}}-\vx_j^{t+\frac{1}{2}})\|^2 =  \phi_\lambda(\vx_j^{t+\frac{1}{2}})$,  $\|\widehat\vx_j^{t+\frac{1}{2}}-\widehat\vx_l^{t+\frac{1}{2}}\|^2\leq \frac{1}{(1-\lambda L)^2}\|\vx_j^{t+\frac{1}{2}}-\vx_l^{t+\frac{1}{2}}\|^2$ from Lemma \ref{lem:prox_diff},  $\frac{1}{(1-\lambda L)^2}\leq 2$ and $ L \leq  \frac{1}{4\lambda}$.
For the first term on the right hand side of \eqref{eq:phi_update_0}, with $\sum_{i=1}^n \W_{ji}=1$, we have
\begin{align}
 \sum_{i=1}^n   \sum_{j=1}^n \W_{ji} \phi_\lambda(\vx_j^{t+\frac{1}{2}}) = &~ 
 \sum_{i=1}^n \phi_\lambda(\vx_i^{t+\frac{1}{2}}) 
\leq \sum_{i=1}^n \phi_\lambda( \vx_i^{t})  + { \frac{1}{2\lambda}} \|\widehat\X^{t}-\X^{t+\frac{1}{2}}\|^2  - { \frac{1}{2\lambda}} \|\widehat\X^t - \X^t\|^2, \label{eq:phi_lambda} 
\end{align}
where we have used
$
\phi_\lambda(\vx_i^{t+\frac{1}{2}})  
\leq  \phi(\widehat\vx_i^{t})+\frac{1}{2\lambda} \|\widehat\vx_i^{t}-\vx_i^{t+\frac{1}{2}}\|^2$  and $\phi_\lambda(\vx_i^{t}) = \phi( \widehat\vx_i^{t}) + \frac{1}{2\lambda} \|\widehat\vx_i^{t}-\vx_i^t\|$.
For the second term on the right hand side of \eqref{eq:phi_update_0}, with Lemma \ref{lem:prox_diff} and \eqref{eq:x_half_update}, we have
\begin{align}
&~\sum_{i=1}^n\sum_{j=1}^{n-1}\sum_{l=j+1}^n \W_{ji}\W_{li}\|\vx_j^{t+\frac{1}{2}}-\vx_l^{t+\frac{1}{2}}\|^2  = \sum_{i=1}^n\sum_{j=1}^{n-1}\sum_{l=j+1}^n \W_{ji}\W_{li}\|\prox_{\eta r}(\vx_j^{t}-\eta\vy_j^{t})-\prox_{\eta r}(\vx_l^{t}-\eta\vy_l^t)\|^2 \nonumber\\
 \leq &~  \sum_{i=1}^n\sum_{j=1}^{n-1}\sum_{l=j+1}^n \W_{ji}\W_{li}\|(\vx_j^{t}-\eta\vy_j^{t})-(\vx_l^{t}-\eta\vy_l^t)\|^2 \nonumber\\
 = &~  \sum_{i=1}^n\sum_{j=1}^{n-1}\sum_{l=j+1}^n \W_{ji}\W_{li}\|(\vx_j^{t}-\eta\vy_j^{t})-(\bar\vx^{t}-\eta\bar\vy^t)+(\bar\vx^{t}-\eta\bar\vy^t)-(\vx_l^{t}-\eta\vy_l^t)\|^2 \nonumber\\
\leq&~ 2\sum_{i=1}^n\sum_{j=1}^{n-1}\sum_{l=j+1}^n \W_{ji}\W_{li}\|(\vx_j^{t}-\eta\vy_j^{t})-(\bar\vx^{t}-\eta\bar\vy^t)\|^2 + 2\sum_{i=1}^n\sum_{j=1}^{n-1}\sum_{l=j+1}^n \W_{ji}\W_{li}\|(\bar{\vx}^{t}-\eta\bar{\vy}^{t})-(\vx_l^{t}-\eta\vy_l^t)\|^2 \nonumber\\
\leq&~ 2\sum_{i=1}^n\sum_{j=1}^{n-1}  \W_{ji} \|(\vx_j^{t}-\eta\vy_j^{t})-(\bar\vx^{t}-\eta\bar\vy^t)\|^2 + 2\sum_{i=1}^n \sum_{l=2}^n \W_{li}\|(\bar{\vx}^{t}-\eta\bar{\vy}^{t})-(\vx_l^{t}-\eta\vy_l^t)\|^2 \nonumber \\
\leq&~4 \sum_{j=1}^{n}  \|(\vx_j^{t}-\eta\vy_j^{t})-(\bar\vx^{t}-\eta\bar\vy^t)\|^2  
\leq  8   \|\X^{t}_\perp\|^2+ 8\eta^2  \|\Y^{t}_\perp\|^2. \label{eq:2_3}
\end{align}
With \eqref{eq:phi_lambda} and \eqref{eq:2_3}, summing up \eqref{eq:phi_update_0} from $i=1$ to $n$ gives

\begin{align*}
\sum_{i=1}^n  \phi_\lambda(\vx_i^{t+1}) \leq &~  \sum_{i=1}^n \phi_\lambda( \vx_i^{t}) +{ \frac{1}{2\lambda} }\|\widehat\X^{t}-\X^{t+\frac{1}{2}}\|^2 - { \frac{1}{2\lambda} } \|\widehat\X^t - \X^t\|^2
 +{ \frac{2}{\lambda} }\left( \|\X^{t}_\perp\|^2+  \eta^2  \|\Y^{t}_\perp\|^2 \right).
\end{align*}
Now taking the expectation on the above inequality and using \eqref{eq:hatx_xprox}, we have
\begin{align*}
\sum_{i=1}^n  \bbE\big[\phi_\lambda(\vx_i^{t+1}) \big] \leq &~  \sum_{i=1}^n \bbE\big[\phi_\lambda( \vx_i^{t}) \big] - \frac{1}{2\lambda} \bbE\big[ \|\widehat\X^t - \X^t\|^2\big]
 + \frac{2}{\lambda} \bbE\big[ \|\X^{t}_\perp\|^2+  \eta^2  \|\Y^{t}_\perp\|^2 \big]\\
 &~ \hspace{-2cm}+\frac{1}{2\lambda} \left(\textstyle 4 \bbE\big[\|\X^t_\perp\|^2\big]   +  \left(\textstyle 1-\frac{\eta}{2\lambda} \right) \bbE\big[\| \widehat\X^t - \X^t\|^2\big] +4\eta^2 \bbE\big[\|\Y^t_\perp\|^2\big] + 2\eta^2\sigma^2 \right).
\end{align*}
Combining like terms in the inequality above gives \eqref{eq:phi_update}.
\end{proof}

With Lemmas \ref{lem:XI_J}, \ref{lem:YI_J} and \ref{lem:weak_convex}, we are ready to prove Theorem \ref{thm:sec2}. We build the following Lyapunov function:
\begin{align*}
    \V^t = z_1 \bbE[\|\X^t_\perp\|^2]  +z_2\bbE[\|\Y^t_\perp\|^2]  +z_3\sum_{i=1}^n \bbE[ \phi_\lambda( \vx_i^{t})], 
\end{align*}
where $z_1, z_2, z_3 \geq 0$ will be determined later.

\subsection*{Proof of Theorem \ref{thm:sec2}.}
\begin{proof}
Denote 
\begin{align*}
    \Phi^t = \sum_{i=1}^n \bbE[ \phi_\lambda( \vx_i^{t})],\quad \Omega_0^t = \bbE[\|\widehat\X^{t}-\X^{t}\|^2],\quad 
    \Omega^t = \left(\bbE[\|\X^t_\perp\|^2], \bbE[\|\Y^t_\perp\|^2], \Phi^t\right)^\top.
\end{align*}
Then Lemmas \ref{lem:XI_J}, \ref{lem:YI_J} and \ref{lem:weak_convex} imply $\Omega^{t+1} \leq \A\Omega^t + \vb \Omega_0^t + \vc \sigma^2$,
where 
\begin{align*}
    \A = \begin{pmatrix}
             \frac{1+\rho^2}{2}  &~  \frac{2\rho^2}{1-\rho^2}\eta^2 &~ 0\\
             \frac{48\rho^2  L^2 }{1-\rho^2 }   &~\frac{3+\rho^2}{4}  &~ 0 \\
             \frac{4}{\lambda}  &~  \frac{4}{\lambda}\eta^2 &~ 1
        \end{pmatrix}, \quad 
\vb = 
 \begin{pmatrix}
     0 \\
    \frac{12\rho^2 L^2  }{1-\rho^2 }  \\
     - \frac{\eta}{4\lambda^2}
 \end{pmatrix}, \quad
\vc = 
 \begin{pmatrix}
        0      \\
        6n   \\    
         \frac{\eta^2}{\lambda}
 \end{pmatrix}.
\end{align*}
For any $\vz = (z_1, z_2, z_3)^\top \geq \0$, We have
\begin{align*}
    \vz^\top \Omega^{t+1} \leq \vz^\top \Omega^{t}+ (\vz^\top \A-\vz^\top)\Omega^t +\vz^\top \vb \Omega_0^t + \vz^\top\vc \sigma^2. 
\end{align*}
Take $$z_1=\frac{10}{1-\rho^2},\ z_2=\left(\frac{80\rho^2}{(1-\rho^2)^3} + \frac{16}{1-\rho^2}\right)\eta^2,\  z_3 = \lambda.$$ We have
$
\vz^\top \A-\vz^\top =  \begin{pmatrix}
       \frac{48\rho^2  L^2 }{1-\rho^2 }z_2-1,
      0, 
      0
 \end{pmatrix}.
$
Note $z_2 \leq \frac{96}{(1-\rho^2)^3}\eta^2$. Thus
\begin{align*}
\vz^\top \A-\vz^\top \leq  \begin{pmatrix} \textstyle
      \frac{4608\rho^2  L^2 }{(1-\rho^2)^4 }\eta^2-1,
      0, 
      0
 \end{pmatrix}, \ 
\vz^\top\vb \leq   \textstyle  \frac{1152\rho^2 L^2  }{(1-\rho^2)^4 }\eta^2 -  \frac{\eta}{4\lambda}, \ 
\vz^\top\vc \leq  \textstyle  \Big(  \textstyle  \frac{576n  }{(1-\rho^2)^3} + 1\Big)\eta^2 \leq \frac{577n}{(1-\rho^2)^3} \eta^2.
\end{align*}

With $\eta\leq \frac{(1-\rho^2)^4}{96\rho  L}$ and $\lambda \leq \frac{1}{96\rho  L}$,  we have
$ \vz^\top \A-\vz^\top \leq  (-\frac{1}{2}, 0, 0 )^\top$ and 
$
\vz^\top\vb \leq   
\left(12\rho L - \frac{1}{8\lambda}\right)\eta -  \frac{\eta}{8\lambda}
\leq -\frac{\eta}{8\lambda}$. Thus 
\begin{align}
    \vz^\top \Omega^{t+1} \leq  \textstyle  \vz^\top \Omega^{t} -\frac{1}{2}\bbE[\|\X^t_\perp\|^2] -\frac{\eta}{8\lambda} \Omega_0^t + \frac{577n}{(1-\rho^2)^3} \eta^2 \sigma^2.\label{eq:l_fun}
\end{align}
Hence, summing up \eqref{eq:l_fun} for $t=0,1,\ldots,T-1$ gives 
\begin{align}\label{eq:avg-Omega}
   \frac{1}{\lambda T}\sum_{t=0}^{T-1} \Omega_0^t +\frac{4}{\eta T}\sum_{t=0}^{T-1} \bbE[\|\X^t_\perp\|^2] \leq  \textstyle  \frac{8}{\eta T}  \left(\vz^\top \Omega^0 - \vz^\top \Omega^{T}\right) + \frac{577n}{(1-\rho^2)^3} 8\eta\sigma^2 .
\end{align}
From $\vy_i^{-1} =\0, \nabla F_i(\vx_i^{-1},\xi_i^{-1}) = \0, \vx_i^0 = \vx^0, \forall\, i \in \hN$, we have  
\begin{align}
    \|\X^0_\perp\|^2 = 0, \quad  \|\Y^0_\perp\|^2  = \|\nabla \bF^0(\I-\J)\|^2, \quad \Phi^0=n \phi_\lambda(\vx^0). \label{eq:initial_thm2}
\end{align}
From Assumption \ref{assu:prob}, $\phi$ is lower bounded and thus $\phi_\lambda $ is also lower bounded, i.e., there is a constant $\phi_\lambda^*$ satisfying $\phi_\lambda^* = \min_{\vx} \phi_\lambda(\vx) > -\infty$. Thus 
\begin{align}
    \Phi^T \geq n \phi_\lambda^*.\label{eq:end_thm2}
\end{align}
With \eqref{eq:initial_thm2}, \eqref{eq:end_thm2}, and the nonnegativity of $ \bbE[\|\X^T_\perp\|^2]$ and $ \bbE[\|\Y^T_\perp\|^2]$, we have
\begin{align}
\textstyle
\vz^\top \Omega^0 - \vz^\top \Omega^{T} \le 
\frac{96 \eta^2}{(1-\rho^2)^3} \bbE[ \|\nabla \bF^0(\I-\J)\|^2] + \lambda n \phi_\lambda(\vx^0) -\lambda n  \phi_\lambda^*. \label{eq:Omega0_OmegaT}
\end{align}
By the convexity of the Frobenius norm and  \eqref{eq:Omega0_OmegaT}, we obtain from \eqref{eq:avg-Omega} that
\begin{align*} 
    &~ \frac{1}{\lambda^2n} \bbE\big[\|\widehat\X^{\tau}-\X^{\tau}\|^2\big] +\frac{4}{n \lambda \eta}\bbE\big[\|\X^\tau_\perp\|^2\big] \leq  \frac{1}{\lambda^2n T}\sum_{t=0}^{T-1}  \bbE\big[\|\widehat\X^{t}-\X^{t}\|^2\big] +\frac{4}{n \lambda \eta T}\sum_{t=0}^{T-1} \bbE\big[\|\X^t_\perp\|^2\big]  \nonumber \\ 
   \leq &~ \textstyle  \frac{8\left( \phi_\lambda(\vx^0) - \phi_\lambda^*\right)}{ \eta T}  + \frac{4616 \eta}{\lambda(1-\rho^2)^3} \sigma^2 \textstyle  + \frac{768\eta  \bbE\left[ \|\nabla \bF^0(\I-\J)\|^2\right]}{n\lambda T(1-\rho^2)^3}.
\end{align*}
 Note $\|\nabla \phi_\lambda (\vx_i^\tau)\|^2 = \frac{\|\vx_i^\tau-\widehat\vx_i^\tau\|^2}{\lambda^{2}}$ from Lemma \ref{lem:xhat_x}, we finish the proof. 
\end{proof}

%% file: appendix_proof_sect3.tex
\section{Convergence Analysis for CDProxSGT} \label{sec:proof_CDProxSGT}
In this section, we analyze the convergence rate of CDProxSGT. Similar to the analysis of DProxSGT, we  establish a Lyapunov function that involves consensus errors and the Moreau envelope. But due to the compression, 
 compression errors $\|\widehat\X^t-\X^t\|$ and  $\|\widehat\Y^t-\Y^t\|$ will occur. Hence, we will also include the two compression errors in our Lyapunov function. 


Again, we can equivalently write a matrix form of the updates 
\eqref{eq:alg3_1}-\eqref{eq:alg3_6} in 
Algorithm \ref{alg:CDProxSGT} as follows: 
\begin{gather}
    \Y^{t-\frac{1}{2}} =  \Y^{t-1} + \nabla \bF^t - \nabla \bF^{t-1}, \label{eq:alg3_1_matrix}\\
    \underline\Y^{t} = \underline\Y^{t-1} + Q_\vy\big[\Y^{t-\frac{1}{2}} - \underline\Y^{t-1}\big], \label{eq:alg3_2_matrix}\\
    \Y^{t} = \Y^{t-\frac{1}{2}} +\gamma_y \underline\Y^{t}(\W-\I), \label{eq:alg3_3_matrix}\\
    \X^{t+\frac{1}{2}} =\prox_{\eta r} \left(\X^t - \eta \Y^{t}\right), \label{eq:alg3_4_matrix}\\
    \underline\X^{t+1} = \underline\X^{t} + Q_\vx\big[\X^{t+\frac{1}{2}} - \underline\X^{t}\big], \label{eq:alg3_5_matrix}\\
    \X^{t+1} = \X^{t+\frac{1}{2}}+\gamma_x\underline\X^{t+1}(\W-\I).\label{eq:alg3_6_matrix}
\end{gather}  
When we apply the compressor to the column-concatenated matrix in \eqref{eq:alg3_2_matrix} and \eqref{eq:alg3_5_matrix}, it means applying the compressor to each column separately, i.e., 
$Q_\vx[\X] = [Q_x[\vx_1],Q_x[\vx_2],\ldots,Q_x[\vx_n]]$.

Below we first analyze the progress by the half-step updates of $\Y$ and $\X$ from $t+1/2$ to $t+1$ in Lemmas \ref{lem:prepare_comp_y} and \ref{lem:Xhat_Xhalf_comp}. Then we bound the one-step consensus error and compression error for $\X$ 
in Lemma \ref{lem:X_consensus_comperror} and for $\Y$ 
in Lemma \ref{lem:Y_consensus_comperror}. The bound of $\bbE[\phi_\lambda(\vx_i^{t+1})]$ after one-step update
is given in \ref{lem:phi_one_step}. Finally, we prove Theorem \ref{thm:sect3thm} by building a Lyapunov function that involves all the five terms.

\begin{lemma} \label{lem:prepare_comp_y} It holds that 
\begin{align}
    \bbE\big[\|\underline\Y^{t+1}-\Y^{t+\frac{1}{2}}\|^2\big]  \leq &~2 \alpha^2\bbE\big[\|\Y^{t} -\underline\Y^{t}\|^2\big] + 
    6 \alpha^2  n \sigma^2 + 4 \alpha^2  L^2 \bbE\big[\|\X^{t+1}-\X^{t}\|^2\big], \label{eq:2.3.2_1}  \\
    \bbE\big[\|\underline\Y^{t+1}-\Y^{t+\frac{1}{2}}\|^2\big]   \leq &~\frac{1+\alpha^2}{2}\bbE\big[\|\Y^{t} -\underline\Y^{t}\|^2\big] + \frac{6 n \sigma^2}{1-\alpha^2}  + \frac{4 L^2}{1-\alpha^2}  \bbE\big[\|\X^{t+1}-\X^{t}\|^2\big]. \label{eq:2.3.2} 
\end{align}
\end{lemma}
\begin{proof}
From \eqref{eq:alg3_1} and \eqref{eq:alg3_2}, we have
\begin{align}
    &~ \bbE\big[\|\underline\Y^{t+1}-\Y^{t+\frac{1}{2}}\|^2\big] =    \bbE\big[\bbE_Q\big[\|Q_\vy\big[\Y^{t+\frac{1}{2}}-\underline\Y^{t}\big]- (\Y^{t+\frac{1}{2}}-\underline\Y^{t})\|^2\big]\big] \nonumber\\
    \leq &~ \alpha^2\bbE\big[\|\Y^{t+\frac{1}{2}}-\underline\Y^{t}\|^2\big] =  \alpha^2\bbE\big[\|\Y^{t} -\underline\Y^{t} +\nabla \bF^{t+1}-\nabla \bF^{t}\|^2\big]\nonumber\\
    \leq &~ \alpha^2(1+\alpha_0)\bbE\big[\|\Y^{t} -\underline\Y^{t}\|^2\big] + \alpha^2(1+\alpha_0^{-1})\bbE\big[\|\nabla \bF^{t+1}-\nabla \bF^{t}\|^2\big] \nonumber\\
    \leq &~ \alpha^2(1+\alpha_0)\bbE\big[\|\Y^{t} -\underline\Y^{t}\|^2\big] + \alpha^2(1+\alpha_0^{-1}) 
    \left(3 n \sigma^2 + 2 L^2 \bbE\big[\|\X^{t+1}-\X^{t}\|^2\big]\right), \label{eq:2.3.2_0}
\end{align}
where the first inequality holds by Assumption \ref{assu:compressor}, 
$\alpha_0$ can be any positive number,
and the last inequality holds by \eqref{eq:y_cons12} which still holds for CDProxSGT. Taking $\alpha_0=1$ in \eqref{eq:2.3.2_0} gives \eqref{eq:2.3.2_1}. Letting 
$\alpha_0=\frac{1-\alpha^2}{2}$ in  \eqref{eq:2.3.2_0}, we obtain $\alpha^2(1+\alpha_0) = (1-(1-\alpha^2))(1+\frac{1-\alpha^2}{2}) \leq \frac{1+\alpha^2}{2}$ and   $\alpha^2(1+\alpha_0^{-1}) \leq \frac{2}{1-\alpha^2}$,  and thus \eqref{eq:2.3.2} follows. 
\end{proof}
 
\begin{lemma} \label{lem:Xhat_Xhalf_comp} 
Let $\eta\leq \lambda \leq\frac{1}{4 L}$. Then
\begin{align}
    \bbE\big[\|\widehat\X^{t}-\X^{t+\frac{1}{2}}\|^2\big]  
        \leq &~ 4\bbE\big[\|\X^t_\perp\|^2\big]   +  \left( 1-\frac{\eta}{2\lambda} \right) \bbE\big[\| \widehat\X^t - \X^t\|^2\big] +4\eta^2 \bbE\big[\|\Y^t_\perp\|^2\big] + 2\eta^2\sigma^2, \label{eq:hatx_xprox_comp}\\
    \bbE\big[\|\underline\X^{t+1}-\X^{t+\frac{1}{2}}\|^2\big]
        \leq &~ 3\alpha^2 \left(\bbE\big[ \|\X^t-\underline\X^{t}\|^2\big] +  \bbE\big[\|\X^{t+\frac{1}{2}}-\widehat{\X}^t\|^2\big]+ \bbE\big[\|\widehat{\X}^t - \X^t\|^2\big]\right), \label{eq:X_-X_1}\\ 
    \bbE\big[\|\underline\X^{t+1}-\X^{t+\frac{1}{2}}\|^2\big] \leq
        &~ \frac{16}{1-\alpha^2}\Big( \bbE\big[\|\X^t_\perp\|^2\big]+ \eta^2\bbE\big[\|\Y^t_\perp\|^2\big]\Big) +  \frac{1+\alpha^2}{2} \bbE\big[\|\X^t-\underline\X^{t}\|^2\big]    \nonumber\\ 
        &~ +\frac{8}{1-\alpha^2}\left( \bbE\big[\| \widehat\X^t - \X^t\|^2\big] +\eta^2\sigma^2\right). \label{eq:2.2.2}
\end{align}
Further, if $\gamma_x\leq \frac{2\sqrt{3}-3}{6\alpha}$, then 
\begin{align}
  \bbE\big[\| \X^{t+1} - \X^{t}\|^2\big] \leq 
    &~ 30\bbE\big[\|\X^t_\perp\|^2\big] +4\sqrt{3} \alpha \gamma_x \bbE\big[\|\X^t-\underline\X^{t}\|^2\big] +16\eta^2 \bbE\big[\|\Y^t_\perp\|^2\big] \nonumber \\
    &~ + 8\bbE\big[\| \widehat\X^t - \X^t\|^2\big] +  8\eta^2\sigma^2. \label{eq:2.2.3}    
\end{align}
\end{lemma}
\begin{proof}
The proof of \eqref{eq:hatx_xprox_comp} is the same as that of Lemma \ref{lem:Xhat_Xhalf} because \eqref{eq:alg3_4} and \eqref{eq:x_y_mean} are the same as \eqref{eq:x_half_update} and \eqref{eq:x_y_mean}.

For $\underline\X^{t+1}-\X^{t+\frac{1}{2}}$, we have from \eqref{eq:alg3_5} that
\begin{align}
    &~ \bbE\big[\|\underline\X^{t+1}-\X^{t+\frac{1}{2}}\|^2\big] = \bbE\big[\bbE_Q\big[\|  Q_\vx\big[\X^{t+\frac{1}{2}} - \underline\X^{t}\big] -(\X^{t+\frac{1}{2}}-\underline\X^{t})\|^2\big]\big] \nonumber \\
   \leq &~ \alpha^2 \bbE\big[\|\X^{t+\frac{1}{2}}-\underline\X^{t}\|^2\big] = \alpha^2 \bbE\big[\|\X^{t+\frac{1}{2}}-\widehat{\X}^t+\widehat{\X}^t - \X^t+\X^t-\underline\X^{t}\|^2\big] \nonumber \\ 
 \le & ~ \alpha^2(1+\alpha_1)\bbE\big[ \|\X^t-\underline\X^{t}\|^2\big] +  \alpha^2(1+\alpha_1^{-1})\bbE\big[\|\X^{t+\frac{1}{2}}-\widehat{\X}^t + \widehat{\X}^t - \X^t\|^2\big]  \nonumber \\ 
\leq &~ \alpha^2(1+\alpha_1)\bbE\big[ \|\X^t-\underline\X^{t}\|^2\big] + 2\alpha^2(1+\alpha_1^{-1})\bbE\big[\|\X^{t+\frac{1}{2}}-\widehat{\X}^t\|^2\big]+2\alpha^2(1+\alpha_1^{-1})\bbE\big[\|\widehat{\X}^t - \X^t\|^2\big], \label{eq:X_-X_0} 
\end{align}
where $\alpha_1$ can be any positive number.
Taking $\alpha_1 = 2$ in \eqref{eq:X_-X_0} gives  \eqref{eq:X_-X_1}.
Taking $\alpha_1 = \frac{1-\alpha^2}{2}$ in \eqref{eq:X_-X_0} and plugging \eqref{eq:hatx_xprox_comp} give \eqref{eq:2.2.2}.

About $\bbE[\| \X^{t+1} - \X^{t}\|^2]$, similar to \eqref{eq:Xplus1-X}, we have from \eqref{eq:compX_hatW} that
\begin{align}
&~   \bbE\big[\| \X^{t+1} - \X^{t}\|^2\big] = \bbE\big[\|\X^{t+\frac{1}{2}}\widehat\W_x  - \X^{t} + \gamma_x(\underline\X^{t+1}-\X^{t+\frac{1}{2}})(\W-\I)\|^2\big] \nonumber \\
\leq&~(1+\alpha_2)  \bbE\big[\|\X^{t+\frac{1}{2}}\widehat\W_x-\X^t\|^2\big] + (1+\alpha_2^{-1}) \bbE\big[\|\gamma_x(\underline\X^{t+1}-\X^{t+\frac{1}{2}})(\W-\I)\|^2\big]\nonumber\\
 \overset{\eqref{eq:Xplus1-X}, \eqref{eq:X_-X_1}}\leq &~ (1+\alpha_2) \left(  3\bbE\big[\|\X^{t+\frac{1}{2}}-\widehat\X^t \|^2\big] +3\bbE\big[\|\widehat\X^t-\X^t \|^2\big] + 12 \bbE\big[\|\X^t_\perp\|^2\big]\right) \nonumber \\
&~ + (1+\alpha_2^{-1})4\gamma_x^2 \cdot 3\alpha^2 \left(  \bbE\big[\|\X^{t+\frac{1}{2}}-\widehat\X^t \|^2\big] +  \bbE\big[\|\widehat\X^t-\X^t \|^2\big] +  \bbE\big[\|\X^t-\underline\X^{t}\|^2\big]\right) \nonumber \\
\leq &~ 4\bbE\big[\|\X^{t+\frac{1}{2}}-\widehat\X^t \|^2\big]  + 4 \bbE\big[\|\widehat\X^t-\X^t \|^2\big] +  14\bbE\big[\|\X^t _\perp\|^2\big]  + 4\sqrt{3} \alpha \gamma_x \bbE\big[\|\X^t-\underline\X^{t}\|^2\big], \nonumber
\end{align}
where in the first inequality $\alpha_2$ could be any positive number,  in the second inequality we use \eqref{eq:X_-X_1},
and in the last inequality we take $\alpha_2 = 2\gamma_x \alpha$ and thus with  $\gamma_x\leq \frac{2\sqrt{3}-3}{6\alpha}$, it holds
$ 3(1+\alpha_2) +12\gamma_x^2\alpha^2(1+\alpha_2^{-1}) = 3(1+2\gamma_x\alpha)^2  \leq 4$,
   $12(1+\alpha_2)\leq 
   8\sqrt{3}\leq 14$,  
$(1+\alpha_2^{-1})4\gamma_x^2\cdot3\alpha^2 \leq 
4\sqrt{3} \alpha \gamma_x$.
Then plugging \eqref{eq:hatx_xprox_comp} into the inequality above, we obtain \eqref{eq:2.2.3}.
\end{proof}

\begin{lemma}\label{lem:X_consensus_comperror} 
Let $\eta\leq \lambda \leq\frac{1}{4 L}$ and $\gamma_x\leq \min\{\frac{ (1-\widehat\rho_x^2)^2}{60\alpha},   \frac{1-\alpha^2}{25}\}$.
Then the consensus error and compression error of $\X$ can be bounded  by
\begin{align} 
    \bbE\big[\|\X^{t+1}_\perp\|^2\big] 
    \leq  &~
    \frac{3+\widehat\rho_x^2}{4} \bbE\big[\|\X^t_\perp\|^2\big]  + 2\alpha \gamma_x (1-\widehat\rho_x^2) \bbE\big[\|\X^t-\underline\X^{t}\|^2\big]
   + \frac{9}{4(1-\widehat\rho_x^2)}\eta^2  \bbE\big[\|\Y^t_\perp\|^2\big] \nonumber\\
    &~  + 4\alpha \gamma_x (1-\widehat\rho_x^2)\bbE\big[\| \widehat\X^t - \X^t\|^2\big]  + 4 \alpha \gamma_x (1-\widehat\rho_x^2)\eta^2\sigma^2,  \label{eq:2.4.1}\\
 \bbE\big[\|\X^{t+1}-\underline\X^{t+1}\|^2\big] 
  \leq &~ \frac{21}{1-\alpha^2} \bbE\big[\|\X^t_\perp\|^2\big] + \frac{3+\alpha^2}{4}\bbE\big[\|\X^t-\underline\X^{t}\|^2\big] +\frac{21}{1-\alpha^2} \eta^2 \bbE\big[\|\Y^t_\perp\|^2\big]\nonumber\\
&~ + \frac{11}{1-\alpha^2}  \bbE\big[\| \widehat\X^t - \X^t\|^2\big]  + \frac{11}{1-\alpha^2} \eta^2\sigma^2. \label{eq:2.5.1}
\end{align}
\end{lemma}
\begin{proof}
First, let us consider the consensus error of $\X$. 
With the update \eqref{eq:compX_hatW}, we have
\begin{align}
     \bbE\big[\|\X^{t+1}_\perp\|^2\big] \leq &~   (1+\alpha_3)\bbE\big[\|\X^{t+\frac{1}{2}}\widehat\W_x (\I- \J)\|^2\big] +(1+\alpha_3^{-1}) \bbE\big[\|\gamma_x(\underline\X^{t+1}-\X^{t+\frac{1}{2}})(\W-\I)\|^2\big], \nonumber\\
      \leq &~ (1+\alpha_3)\bbE\big[\|\X^{t+\frac{1}{2}}(\widehat\W_x - \J)\|^2\big] + (1+\alpha_3^{-1})4\gamma_x^2\bbE\big[\|\underline\X^{t+1}-\X^{t+\frac{1}{2}}\|^2\big], \label{eq:XComp_consensus0}
\end{align}
where $\alpha_3$ is any positive number, and $\|\W-\I\|_2\leq 2$ is used.
The first term in the right hand side of \eqref{eq:XComp_consensus0} can be processed similarly as the non-compressed version in Lemma \ref{lem:XI_J} by replacing $\W$ by $\widehat\W_x$, namely,
\begin{align}
      \bbE\big[\|\X^{t+\frac{1}{2}} (\widehat\W_x-\J)\|^2\big]
  \leq &~  \textstyle  \frac{1+\widehat\rho^2_x}{2} \bbE\big[\|\X^{t}_\perp\|^2\big]+ \frac{2\widehat\rho^2_x  \eta^2 }{1-\widehat\rho^2_x}  \bbE\big[\| \Y^{t}_\perp \|^2\big]. \label{eq:XComp_consensus1}
\end{align}  
Plugging \eqref{eq:XComp_consensus1} and  \eqref{eq:X_-X_1}  into \eqref{eq:XComp_consensus0} gives
\begin{align*}
     &~ \bbE\big[\|\X^{t+1}_\perp\|^2\big] \leq ~ (1+\alpha_3)\left( \textstyle  \frac{1+\widehat\rho^2_x}{2} \bbE\big[\|   \X^{t}_\perp \|^2\big]+ \frac{2\widehat\rho^2_x  \eta^2 }{1-\widehat\rho^2_x}  \bbE\big[\| \Y^{t}_\perp \|^2\big]\right) \\
    &~ + (1+\alpha_3^{-1})12 \alpha^2 \gamma_x^2 \left(\bbE\big[ \|\X^t-\underline\X^{t}\|^2\big] +  \bbE\big[\|\X^{t+\frac{1}{2}}-\widehat{\X}^t\|^2\big]+ \bbE\big[\|\widehat{\X}^t - \X^t\|^2\big]\right)\\
    \overset{\eqref{eq:hatx_xprox_comp}}{\leq} &~
     \left(  \textstyle \frac{1+\widehat\rho_x^2}{2}(1+\alpha_3) + 48 \alpha^2 \gamma_x^2  (1+\alpha_3^{-1})   \right) \bbE\big[\|\X^t_\perp\|^2\big] \nonumber\\
    &~+ 12\alpha^2 \gamma_x^2  (1+\alpha_3^{-1}) \bbE\big[\|\X^t-\underline\X^{t}\|^2\big] +\left( \textstyle \frac{2\widehat\rho_x^2}{1-\widehat\rho_x^2}(1+\alpha_3) +48 \alpha^2 \gamma_x^2  (1+\alpha_3^{-1})\right)\eta^2  \bbE\big[\|\Y^t_\perp\|^2\big] \\
&~ +24 \alpha^2 \gamma_x^2  (1+\alpha_3^{-1}) \bbE\big[\| \widehat\X^t - \X^t\|^2\big]  +24 \alpha^2 \gamma_x^2  (1+\alpha_3^{-1})\eta^2\sigma^2.
\end{align*}  
Let $\alpha_3 = \frac{7\alpha\gamma_x}{1-\widehat\rho_x^2}$ and $\gamma_x\leq \frac{(1-\widehat\rho_x^2)^2}{60\alpha}$. 
Then $\alpha^2 \gamma_x^2  (1+\alpha_3^{-1})=\alpha\gamma_x (\alpha\gamma_x+\frac{1-\widehat\rho_x^2}{7})\leq \alpha\gamma_x (\frac{ (1-\widehat\rho_x^2)^2}{60}+\frac{1-\widehat\rho_x^2}{7})\leq   \frac{\alpha\gamma_x (1-\widehat\rho_x^2)}{6}$
and 
\begin{align*}
 &~  \textstyle  \frac{1+\widehat\rho_x^2}{2}(1+\alpha_3) + 48 \alpha^2 \gamma_x^2  (1+\alpha_3^{-1}) = \frac{1+\widehat\rho_x^2}{2} +  48 \alpha^2 \gamma_x^2  +  \frac{7\alpha\gamma_x}{1-\widehat\rho_x^2} + \frac{48\alpha\gamma_x(1-\widehat\rho_x^2)}{7} \\ 
 \leq&~  \textstyle  \frac{1+\widehat\rho_x^2}{2} +  \frac{48}{60^2}(1-\widehat\rho_x^2)^4 + \frac{7}{60}(1-\widehat\rho_x^2) + \frac{7}{60}(1-\widehat\rho_x^2)^3\leq \frac{1+\widehat\rho_x^2}{2} + \frac{ 1-\widehat\rho_x^2}{4} = \frac{3+\widehat\rho_x^2}{4},\\
 &~ \textstyle  \frac{2\widehat\rho_x^2}{1-\widehat\rho_x^2}(1+\alpha_3) + 48 \alpha^2 \gamma_x^2  (1+\alpha_3^{-1}) =  \frac{2\widehat\rho_x^2}{1-\widehat\rho_x^2}  + 48 \alpha^2 \gamma_x^2  +  
 \frac{2\widehat\rho_x^2}{1-\widehat\rho_x^2} \frac{7 \alpha \gamma_x }{1-\widehat\rho_x^2} +  \frac{48\alpha\gamma_x(1-\widehat\rho_x^2)}{7}\\
  \leq &~ \textstyle \frac{1}{1-\widehat\rho_x^2} \left(
  2\widehat\rho_x^2 + \frac{48}{60^2} (1-\widehat\rho_x^2) + \frac{14\widehat\rho_x^2}{60} + \frac{7}{60}(1-\widehat\rho_x^2)
  \right)  \leq  
  \frac{1}{1-\widehat\rho_x^2} \left(
  2\widehat\rho_x^2 + \frac{48}{60^2}  + \frac{7}{60}
  \right) \leq  \frac{9}{4(1-\widehat\rho_x^2)}. 
\end{align*}
Thus
\eqref{eq:2.4.1} holds. 

Now let us consider the compression error of $\X$.
By \eqref{eq:alg3_6}, we have
\begin{align}
 &~\bbE\big[\|\X^{t+1}-\underline\X^{t+1}\|^2\big]
 =  \bbE\big[\|(\underline\X^{t+1} - \X^{t+\frac{1}{2}}) \big(\gamma_x(\W-\I) -\I\big) + \gamma_x \X^{t+\frac{1}{2}} (\I-\J) (\W-\I) \|^2\big] \nonumber\\
\leq&~ (1+\alpha_4) (1+2\gamma_x)^2  \bbE\big[\|\underline\X^{t+1}-\X^{t+\frac{1}{2}}\|^2\big]  + (1+\alpha_4^{-1})4 \gamma_x^2  \bbE\big[\|\X^{t+\frac{1}{2}}_\perp\|^2\big],\label{eq:2.5.1.0}
\end{align}
where we have used $\J\W=\J$ in the equality, 
$\|\gamma_x (\W-\I) -\I\|_2\leq \gamma_x\|\W-\I\|_2+\|\I\|_2\leq 1+2\gamma_x$ and $\|\W-\I\|_2\leq 2$ in the inequality, and $\alpha_4$ can be any positive number. For the second term in the right hand side of \eqref{eq:2.5.1.0}, we have
\begin{align}
    \|\X^{t+\frac{1}{2}}_\perp\|^2 \overset{\eqref{eq:alg3_4}}{=}&~ \left\|\left(\prox_{\eta r} \left(\X^t - \eta \Y^{t}\right)-\prox_{\eta r} \left(\bar\vx^t - \eta \bar\vy^{t}\right)\1^\top\right)(\I-\J)\right\|^2 \nonumber \\
     \leq&~ 
     \|\X^t_\perp- \eta \Y^{t}_\perp\|^2   
     \leq   2\|\X^t_\perp\|^2+2\eta^2\|\Y^{t}_\perp\|^2, \label{eq:2.2.1}
\end{align}
where we have used $\1^\top(\I-\J)=\0^\top$, $\|\I-\J\|_2\leq 1$, and Lemma \ref{lem:prox_diff}.
Now plugging  \eqref{eq:2.2.2} and \eqref{eq:2.2.1} into \eqref{eq:2.5.1.0} gives
\begin{align*}
    \bbE\big[\|\X^{t+1}-\underline\X^{t+1}\|^2\big]
    \leq  \left( \textstyle (1+\alpha_4^{-1})8\gamma_x^2+(1+\alpha_4)  (1+2\gamma_x)^2\frac{16}{1-\alpha^2}\right) \left( \bbE\big[\|\X^t_\perp\|^2\big] + \eta^2 \bbE\big[\|\Y^t_\perp\|^2\big]\right) \nonumber\\
      \textstyle  + (1+\alpha_4) (1+2\gamma_x)^2\frac{1+\alpha^2}{2}\bbE\big[\|\X^t-\underline\X^{t}\|^2\big]
    +(1+\alpha_4)(1+2\gamma_x)^2\frac{8}{1-\alpha^2} \left(  \bbE\big[\| \widehat\X^t - \X^t\|^2\big]  +  \eta^2\sigma^2\right).
\end{align*}
With $\alpha_4=\frac{1-\alpha^2}{12}$ and $\gamma_x\leq \frac{1-\alpha^2}{25}$, \eqref{eq:2.5.1} holds because $(1+2\gamma_x)^2  
    \leq 1 + \frac{104}{25}\gamma_x  \leq \frac{7}{6}$, $ (1+2\gamma_x)^2\frac{1+\alpha^2}{2}\leq \frac{1+\alpha^2}{2}+\frac{104}{25}\gamma_x\leq \frac{2+\alpha^2}{3}$,  and
\begin{align} 
    (1+\alpha_4) (1+2\gamma_x)^2\frac{1+\alpha^2}{2}  \leq &~ \frac{2+\alpha^2}{3}  + \alpha_4 = \frac{3+\alpha^2}{4}, \label{eq:gamma_x_1}\\
   (1+\alpha_4^{-1}) 8\gamma_x^2+  (1+\alpha_4) (1+2\gamma_x)^2\frac{16}{1-\alpha^2} \leq&~ \frac{13}{1-\alpha^2}\frac{8}{625} + \frac{13}{12} \frac{7}{6} \frac{16}{1-\alpha^2}  \leq \frac{21}{1-\alpha^2}, \label{eq:gamma_x_2}\\
   (1+\alpha_4)(1+2\gamma_x)^2\frac{8}{1-\alpha^2} \leq&~ \frac{13}{12} \frac{7}{6}\frac{8}{1-\alpha^2} \leq \frac{11}{1-\alpha^2}. \nonumber
\end{align}
\end{proof}

\begin{lemma} 
\label{lem:Y_consensus_comperror} Let $\eta\leq \min\{\lambda, \frac{1-\widehat\rho^2_y}{8\sqrt{5} L} \} $, $\lambda \leq\frac{1}{4 L}$, $ \gamma_x\leq \frac{2\sqrt{3}-3}{6\alpha}$, $\gamma_y\leq \min\{\frac{\sqrt{1-\widehat\rho^2_y}}{12\alpha}, \frac{1-\alpha^2}{25}\}$.
Then the consensus error and compression error of $\Y$ can be bounded by
\begin{align}
 \bbE\big[\|\Y^{t+1}_\perp\|^2\big]  \leq &~   \frac{150 L^2  }{1-\widehat\rho^2_y }   \bbE\big[\|\X^t_\perp\|^2\big] + \frac{20\sqrt{3} \alpha\gamma_x  L^2}{1-\widehat\rho^2_y  }  \bbE\big[\|\X^t-\underline\X^{t}\|^2\big]+\frac{3+\widehat\rho^2_y }{4}\bbE\big[\|\Y^t_\perp\|^2\big] \nonumber\\
&~ +\frac{48\alpha^2\gamma_y^2}{1-\widehat\rho^2_y } \bbE\big[\|\Y^{t} -\underline\Y^{t}\|^2\big]  + \frac{40 L^2  }{1-\widehat\rho^2_y  }  \bbE\big[\| \widehat\X^t - \X^t\|^2\big]  + 12n \sigma^2,  \label{eq:2.4.2} \\
\bbE\big[\|\Y^{t+1}-\underline\Y^{t+1}\|^2\big]
\leq &~ \frac{180 L^2}{1-\alpha^2}\bbE\big[\|\X^t_\perp\|^2\big] + \frac{24\sqrt{3}\alpha\gamma_x  L^2}{1-\alpha^2} \bbE\big[\|\X^t-\underline\X^{t}\|^2\big]
+ \frac{3+\alpha^2}{4}\bbE\big[\|\Y^{t} -\underline\Y^{t}\|^2\big] \nonumber\\
&~ +\frac{104\gamma_y^2+ 96\eta^2 L^2}{1-\alpha^2}\bbE\big[\|\Y^{t}(
\I-\J)\|^2\big] + \frac{48 L^2}{1-\alpha^2} \bbE\big[\| \widehat\X^t - \X^t\|^2\big] +  \frac{10 n}{1-\alpha^2} \sigma^2 .\label{eq:2.5.2}
\end{align} 
\end{lemma}

\begin{proof}
First, let us consider the consensus of $\Y$. Similar to \eqref{eq:XComp_consensus0}, we have from the update \eqref{eq:Y_hatW} that
\begin{align}
 \bbE\big[\|\Y^{t+1}_\perp\|^2\big] \leq  (1+\alpha_5)\bbE\big[\|\Y^{t+\frac{1}{2}}(\widehat\W_y-\J)\|^2\big] + (1+\alpha_5^{-1})4\gamma_y^2  \bbE\big[\|\underline\Y^{t+1}-\Y^{t+\frac{1}{2}}\|^2\big], \label{eq:Ycomp_conses0}
\end{align}
where $\alpha_5$ can be any positive number.
Similarly as  \eqref{eq:y_cons1}-\eqref{eq:y_cons2} in the proof of Lemma \ref{lem:YI_J}, we have the bound for the first term on the right hand side of \eqref{eq:Ycomp_conses0} by replacing  $\W$ with $\widehat\W_y$, namely,  
\begin{align}
    \bbE\big[\|\Y^{t+\frac{1}{2}}(\widehat\W_y-\J)\|^2\big] \leq  \textstyle \frac{1+\widehat\rho^2_y}{2} \bbE\big[\|\Y^{t}_\perp \|^2\big]   + \frac{2 \widehat\rho^2_y L^2 }{1-\widehat\rho^2_y } \bbE\big[\| \X^{t+1} - \X^{t}\|^2\big] +  5 \widehat\rho^2_y  n \sigma^2.\label{eq:comp_y_cons220}  
\end{align} 
Plug \eqref{eq:comp_y_cons220} and  \eqref{eq:2.3.2_1} back to \eqref{eq:Ycomp_conses0}, and take $\alpha_5 = \frac{1-\widehat\rho^2_y}{3(1+\widehat\rho^2_y)}$. We have 
\begin{align*} 
   &~ \bbE\big[\|\Y^{t+1}_\perp\|^2\big] 
   \leq  \textstyle \frac{2(2+\widehat\rho^2_y)}{3(1+\widehat\rho^2_y)}\frac{1+\widehat\rho^2_y}{2} \bbE\big[\|\Y^{t}_\perp \|^2\big] 
   + \frac{24\gamma_y^2}{1-\widehat\rho^2_y}  2\alpha^2 \bbE\big[\|\Y^{t} -\underline\Y^{t}\|^2\big]  \nonumber\\
   &~\quad  \textstyle + \frac{24\gamma_y^2}{1-\widehat\rho^2_y}  6\alpha^2 n\sigma^2  + 2\cdot5 \widehat\rho^2_y  n \sigma^2 + \left(  \textstyle \frac{24\gamma_y^2}{1-\widehat\rho^2_y}  4\alpha^2 L^2 + 2\cdot\frac{2 \widehat\rho^2_y L^2 }{1-\widehat\rho^2_y } \right)\bbE\big[\| \X^{t+1} - \X^{t}\|^2\big] \nonumber\\
  \leq &~  \textstyle \frac{2+\widehat\rho^2_y}{3} \bbE\big[\|\Y^{t}_\perp \|^2\big] 
   +  \frac{48\alpha^2\gamma_y^2}{1-\widehat\rho^2_y}  \bbE\big[\|\Y^{t} -\underline\Y^{t}\|^2\big]  +  11  n \sigma^2  +  \frac{5  L^2}{1-\widehat\rho^2_y} \bbE\big[\| \X^{t+1} - \X^{t}\|^2\big] \\
    \leq &~  \textstyle \frac{150 L^2  }{1-\widehat\rho^2_y }   \bbE\big[\|\X^t_\perp\|^2\big] + \frac{20\sqrt{3} L^2}{1-\widehat\rho^2_y }  \alpha \gamma_x \bbE[\|\X^t-\underline\X^{t}\|^2] +  \frac{40 L^2  }{1-\widehat\rho^2_y }  \eta^2\sigma^2  + 11 n \sigma^2  \nonumber\\
&~ \textstyle +\left( \textstyle \frac{2+\widehat\rho^2_y}{3}+ \frac{80 L^2  }{1-\widehat\rho^2_y }  \eta^2\right) \bbE\big[\|\Y^t_\perp\|^2\big]  +\frac{48\alpha^2\gamma_y^2}{1-\widehat\rho^2_y} \bbE\big[\|\Y^{t} -\underline\Y^{t}\|^2\big]  + \frac{40 L^2  }{1-\widehat\rho^2_y}\bbE\big[\| \widehat\X^t - \X^t\|^2\big],
\end{align*}
where the first inequality holds by  $1+\alpha_5 = \frac{2(2+\widehat\rho^2_y)}{3(1+\widehat\rho^2_y)} \leq 2$ and $1+\alpha_5^{-1} = \frac{2(2+\widehat\rho^2_y)}{1-\widehat\rho^2_y}\leq \frac{6}{1-\widehat\rho^2_y}$, 
the second inequality holds by $\gamma_y\leq \frac{\sqrt{1-\widehat\rho^2_y}}{12\alpha}$ and $\alpha^2\leq 1$, and the third equality holds by \eqref{eq:2.2.3}.
By $\frac{80 L^2  }{1-\widehat\rho^2_y}  \eta^2 \leq \frac{1-\widehat\rho^2_y}{4}$ and  $ \frac{40 L^2  }{1-\widehat\rho^2_y} \eta^2\leq \frac{1-\widehat\rho^2_y}{8}\leq 1$ from $\eta\leq  \frac{1-\widehat\rho^2_y}{8\sqrt{5} L} $, we can now obtain \eqref{eq:2.4.2}.

Next let us consider the compression error of $\Y$, similar to \eqref{eq:2.5.1.0}, we have by \eqref{eq:alg3_3} that
\begin{align} 
&~\bbE\big[\|\Y^{t+1}-\underline\Y^{t+1}\|^2\big] 
\leq (1+\alpha_6)(1+2\gamma_y)^2  \bbE\big[\|\underline\Y^{t+1}-\Y^{t+\frac{1}{2}}\|^2\big] +  (1+\alpha_6^{-1})4 \gamma_y^2   \bbE\big[\|\Y^{t+\frac{1}{2}}_\perp\|^2\big], \label{eq:Y_compress_0}
\end{align}
where 
$\alpha_6$ is any positive number. 
For $\bbE\big[\|\Y^{t+\frac{1}{2}}_\perp\|^2\big]$, we have from \eqref{eq:alg3_1} that
\begin{align}
     &~\bbE\big[\|\Y^{t+\frac{1}{2}}_\perp\|^2\big]  =\bbE\big[\|( \Y^{t} + \nabla \bF^{t+1} - \nabla \bF^{t})(\I-\J)\|^2\big]\nonumber \\
     \leq &~ 2\bbE\big[\|\Y^{t}_\perp\|^2\big] +2\bbE\big[\|\nabla \bF^{t+1}-\nabla \bF^{t}\|^2\big] \leq 
     2\bbE\big[\|\Y^{t}_\perp\|^2\big] +6 n \sigma^2 + 4 L^2 \bbE\big[\|\X^{t+1}-\X^{t}\|^2\big], \label{eq:2.3.1}
\end{align} 
where we have used \eqref{eq:y_cons12}.
Plug \eqref{eq:2.3.2} and \eqref{eq:2.3.1} back to \eqref{eq:Y_compress_0} to have
\begin{align*} 
&~\bbE\big[\|\Y^{t+1}-\underline\Y^{t+1}\|^2\big] \leq \textstyle  (1+\alpha_6) (1+2\gamma_y)^2 \frac{1+\alpha^2}{2}\bbE\big[\|\Y^{t} -\underline\Y^{t}\|^2\big] +(1+\alpha_6^{-1})8\gamma_y^2\bbE\big[\|\Y^{t}(
\I-\J)\|^2\big] \\
&~+ \left( \textstyle (1+\alpha_6^{-1})4\gamma_y^2 +(1+\alpha_6)(1+2\gamma_y)^2 \frac{1}{1-\alpha^2} \right)4 L^2 \bbE\big[\|\X^{t+1}-\X^{t}\|^2\big] \\
&~ +  \left( \textstyle  (1+\alpha_6^{-1})4\gamma_y^2 +(1+\alpha_6)(1+2\gamma_y)^2 \frac{1}{1-\alpha^2} \right) 6 n \sigma^2.
\end{align*} 
With $\alpha_6=\frac{1-\alpha^2}{12}$ and $\gamma_y< \frac{1-\alpha^2}{25}$,
like \eqref{eq:gamma_x_1} and \eqref{eq:gamma_x_2}, we have $(1+\alpha_6) (1+2\gamma_y)^2 \frac{1+\alpha^2}{2}\leq \frac{3+\alpha^2}{4}$, $8(1+\alpha_6^{-1})\leq\frac{8\cdot13}{1-\alpha^2} = \frac{104}{1-\alpha^2}  $ and  $ (1+\alpha_6^{-1})4\gamma_y^2 +(1+\alpha_6)(1+2\gamma_y)^2 \frac{1}{1-\alpha^2} \leq \frac{13}{1-\alpha^2}\frac{4}{625}+\frac{13}{12}\frac{7}{6}\frac{1}{1-\alpha^2}\leq \frac{3}{2(1-\alpha^2)}$. Thus
\begin{align*} 
\bbE\big[\|\Y^{t+1}-\underline\Y^{t+1}\|^2\big]
\leq &~  \textstyle \frac{3+\alpha^2}{4}\bbE\big[\|\Y^{t} -\underline\Y^{t}\|^2\big] 
 +\frac{104\gamma_y^2}{1-\alpha^2}\bbE\big[\|\Y^{t}(
\I-\J)\|^2\big]+\frac{6 L^2}{1-\alpha^2} \bbE\big[\|\X^{t+1}-\X^{t}\|^2\big] +  \frac{9n \sigma^2}{1-\alpha^2} \nonumber\\
\leq &~  \textstyle\frac{180 L^2}{1-\alpha^2}\bbE\big[\|\X^t_\perp\|^2\big] + \frac{24\sqrt{3} \alpha\gamma_x  L^2 }{1-\alpha^2} \bbE\big[\|\X^t-\underline\X^{t}\|^2\big] 
+ \frac{3+\alpha^2}{4}\bbE\big[\|\Y^{t} -\underline\Y^{t}\|^2\big] \\
&~ \textstyle +\frac{104\gamma_y^2+ 96\eta^2 L^2}{1-\alpha^2}\bbE\big[\|\Y^{t}(
\I-\J)\|^2\big] + \frac{48 L^2}{1-\alpha^2} \bbE\big[\| \widehat\X^t - \X^t\|^2\big] +  \frac{48 L^2\eta^2+9n}{1-\alpha^2} \sigma^2,
\end{align*}
where the second inequality holds by \eqref{eq:2.2.3}.
By $48 L^2\eta^2\leq n$, we have \eqref{eq:2.5.2} and complete the proof.
\end{proof}
 
\begin{lemma}\label{lem:phi_one_step}
Let $\eta\leq  \lambda \leq\frac{1}{4 L}$ and $   \gamma_x\leq  \frac{1}{6\alpha}$. 
It holds 
\begin{align}
 \sum_{i=1}^n \bbE\big[\phi_\lambda(\vx_i^{t+1})\big]  
\leq&~ \sum_{i=1}^n\bbE\big[ \phi_\lambda( \vx_i^{t})\big] +   \frac{12}{\lambda}\bbE\big[\|\X^t_\perp\|^2\big] + \frac{7\alpha\gamma_x}{\lambda} \bbE\big[\|\X^t-\underline\X^{t}\|^2\big]  +   \frac{12}{\lambda}  \eta^2\bbE\big[\|\Y^t_\perp\|^2\big] \nonumber \\
&~+\frac{1}{\lambda}\left( -\frac{\eta}{4\lambda} + 23\alpha\gamma_x \right) \bbE\big[\|  \widehat\X^{t}- \X^{t} \|^2\big] + \frac{5}{\lambda} \eta^2 \sigma^2. \label{eq:2.7}
\end{align} 
\end{lemma}

\begin{proof}
Similar to \eqref{eq:phi_update_0}, we have 
\begin{align}
&~ \bbE\big[\phi_\lambda(\vx_i^{t+1})\big] \overset{\eqref{eq:x_t_hat}}{=}  \bbE\big[\phi(\widehat\vx_i^{t+1})\big]+\frac{1}{2\lambda}  \bbE\big[\|\widehat\vx_i^{t+1}-\vx_i^{t+1}\|^2\big] \nonumber \\
\overset{ \eqref{eq:compX_hatW}}{\leq} &~  \bbE\bigg[\phi\bigg(\sum_{j=1}^n \big(\widehat\W_x\big)_{ji}\widehat\vx_j^{t+\frac{1}{2}}\bigg)\bigg] +\frac{1}{2\lambda} \bbE\bigg[\bigg\| \sum_{j=1}^n \big(\widehat\W_x\big)_{ji} \big(\widehat\vx_j^{t+\frac{1}{2}}- \vx_j^{t+\frac{1}{2}}\big) - \gamma_x\sum_{j=1}^n \big(\W_{ji}-\I_{ji}\big)\big(\underline\vx_j^{t+1}-\vx_j^{t+\frac{1}{2}}\big)  \bigg\|^2\bigg]  \nonumber\\
\leq&~  \bbE\bigg[\phi\bigg(\sum_{j=1}^n \big(\widehat\W_x\big)_{ji}\widehat\vx_j^{t+\frac{1}{2}}\bigg)\bigg]  + \frac{1+\alpha_7}{2\lambda} \bbE\bigg[\bigg\| \sum_{j=1}^n \big(\widehat\W_x\big)_{ji} \big(\widehat\vx_j^{t+\frac{1}{2}}-\vx_j^{t+\frac{1}{2}}\big)\bigg\|^2\bigg]\nonumber\\
&~ + \frac{1+\alpha_7^{-1}}{2\lambda} \bbE\bigg[\bigg\| \gamma_x  \sum_{j=1}^n\big(\W_{ji}-\I_{ji}\big)\big(\underline\vx_j^{t+1}-\vx_j^{t+\frac{1}{2}}\big)\bigg\|^2\bigg] \nonumber \\
\overset{\mbox{Lemma \ref{lem:weak_convx}}}\leq &~ \sum_{j=1}^n \big(\widehat\W_x\big)_{ji} \bbE\big[\phi( \widehat\vx_j^{t+\frac{1}{2}})\big] + \frac{ L}{2} \sum_{j=1}^{n-1}\sum_{l=j+1}^n  \big(\widehat\W_x\big)_{ji} (\widehat\W_x)_{li}\bbE\big[\|\widehat\vx_j^{t+\frac{1}{2}}-\widehat\vx_l^{t+\frac{1}{2}}\|^2\big] \nonumber \\
&~ + \frac{1+\alpha_7}{2\lambda} \sum_{j=1}^n \big(\widehat\W_x\big)_{ji}\bbE\big[\|  \widehat\vx_j^{t+\frac{1}{2}}-\vx_j^{t+\frac{1}{2}}\|^2\big] + \frac{1+\alpha_7^{-1}}{2\lambda}\gamma_x^2 \bbE\big[\|  \sum_{j=1}^n(\W_{ji}-\I_{ji})(\underline\vx_j^{t+1}-\vx_j^{t+\frac{1}{2}})\|^2\big] \nonumber \\
\leq &~  \sum_{j=1}^n \big(\widehat\W_x\big)_{ji} \bbE\big[\phi_\lambda(\vx_j^{t+\frac{1}{2}})\big] + \frac{1}{4\lambda} \sum_{j=1}^{n-1}\sum_{l=j+1}^n  \big(\widehat\W_x\big)_{ji} (\widehat\W_x)_{li} \bbE\big[\|\vx_j^{t+\frac{1}{2}}-\vx_l^{t+\frac{1}{2}}\|^2\big]  \nonumber \\
&~+ \frac{\alpha_7}{2\lambda} \sum_{j=1}^n \big(\widehat\W_x\big)_{ji}\bbE\big[\|  \widehat\vx_j^{t+\frac{1}{2}}-\vx_j^{t+\frac{1}{2}}\|^2\big] + \frac{1+\alpha_7^{-1}}{2\lambda}\gamma_x^2 \bbE\big[\|  \sum_{j=1}^n(\W_{ji}-\I_{ji})(\underline\vx_j^{t+1}-\vx_j^{t+\frac{1}{2}})\|^2\big]. \label{eq:phi_lambda1}
\end{align}
The same as \eqref{eq:phi_lambda} and \eqref{eq:2_3}, for the first two terms in the right hand side of \eqref{eq:phi_lambda1},  we have 
\begin{align}
  \sum_{i=1}^n \sum_{j=1}^n \big(\widehat\W_x\big)_{ji}  \phi_\lambda(\vx_j^{t+\frac{1}{2}}) \leq   \sum_{i=1}^n \phi_\lambda( \vx_i^{t})  +\frac{1}{2\lambda} \|\widehat\X^{t}-\X^{t+\frac{1}{2}}\|^2  -  \frac{1}{2\lambda}  \|\widehat\X^t - \X^t\|^2,\label{eq:2_2_press}\\
  \sum_{i=1}^n\sum_{j=1}^{n-1}\sum_{l=j+1}^n \big(\widehat\W_x\big)_{ji}(\widehat\W_x)_{li}\|\vx_j^{t+\frac{1}{2}}-\vx_l^{t+\frac{1}{2}}\|^2  
\leq  8   \|\X^{t}_\perp\|^2+ 8\eta^2  \|\Y^{t}_\perp\|^2. \label{eq:2_3_press}
\end{align}
For the last two terms on the right hand side of \eqref{eq:phi_lambda1}, we have
\begin{align}
    &~ \sum_{i=1}^n\sum_{j=1}^n \big(\widehat\W_x\big)_{ji}\bbE\big[\|  \widehat\vx_j^{t+\frac{1}{2}}-\vx_j^{t+\frac{1}{2}}\|^2\big] 
    = \| \widehat\X^{t+\frac{1}{2}}-\X^{t+\frac{1}{2}} \|^2 
     \leq 2 \| \widehat\X^{t+\frac{1}{2}}-\widehat\X^{t} \|^2 +2 \| \widehat\X^{t} - \X^{t+\frac{1}{2}} \|^2 \nonumber \\
     \leq &~ \textstyle \frac{2}{(1-\lambda L)^2}  \|  \X^{t+\frac{1}{2}}
     - \X^{t} \|^2 +2 \| \widehat\X^{t} - \X^{t+\frac{1}{2}} \|^2 
     \leq  10 \|  \X^{t+\frac{1}{2}}- \widehat\X^{t} \|^2+ 8 \|  \widehat\X^{t}- \X^{t} \|^2, \label{eq:X_-X2}\\
      &~ \sum_{i=1}^n \bbE\big[\|  \sum_{j=1}^n(\W_{ji}-\I_{ji})(\underline\vx_j^{t+1}-\vx_j^{t+\frac{1}{2}})\|^2\big] = \bbE\big[\|(\underline\X^{t+1}-\X^{t+\frac{1}{2}})(\W-\I)\|^2\big]\leq
    4\bbE\big[\|\underline\X^{t+1}-\X^{t+\frac{1}{2}}\|^2\big]\nonumber \\
    \leq &~ 12\alpha^2 \left(\bbE\big[ \|\X^t-\underline\X^{t}\|^2\big] +  \bbE\big[\|\X^{t+\frac{1}{2}}-\widehat{\X}^t\|^2\big]+ \bbE\big[\|\widehat{\X}^t - \X^t\|^2\big]\right), \label{eq:X_-X1}
\end{align}
where   \eqref{eq:X_-X2} holds by Lemma \ref{lem:prox_diff} and $\frac{1}{(1-\lambda L)^2}\leq 2$, and \eqref{eq:X_-X1} holds by \eqref{eq:X_-X_1}.

Sum up \eqref{eq:phi_lambda1}  for $t=0,1,\ldots,T-1$ and take $\alpha_7 =\alpha\gamma_x$.
Then with \eqref{eq:2_2_press}, \eqref{eq:2_3_press},  \eqref{eq:X_-X2}  and \eqref{eq:X_-X1}, we have
\begin{align*}
 \sum_{i=1}^n \bbE\big[\phi_\lambda(\vx_i^{t+1})\big]
\leq  &  ~\sum_{i=1}^n \bbE\big[\phi_\lambda( \vx_i^{t}) \big]  + \frac{2}{\lambda}\left( \bbE\big[\|\X^{t}_\perp\|^2\big] + \eta^2  \bbE\big[\|\Y^{t}_\perp\|^2\big]\right)  +\textstyle \frac{6\alpha\gamma_x+6\alpha^2\gamma_x^2}{\lambda} \bbE\big[\|\X^t-\underline\X^{t}\|^2\big] \nonumber \\
&~ +  \frac{1}{\lambda}\left( \textstyle  \frac{1}{2}+11\alpha\gamma_x +6\alpha^2\gamma_x^2\right) \bbE\big[\|  \X^{t+\frac{1}{2}}- \widehat\X^{t} \|^2\big]+ \frac{1}{\lambda}\left( \textstyle  -\frac{1}{2}+10\alpha\gamma_x +6\alpha^2\gamma_x^2\right) \bbE\big[\|  \widehat\X^{t}- \X^{t} \|^2\big]\\
\leq  & ~ \sum_{i=1}^n\bbE\big[ \phi_\lambda( \vx_i^{t})\big]+ \frac{2}{\lambda}\left( \bbE\big[\|\X^{t}_\perp\|^2\big] +  \eta^2  \bbE\big[\|\Y^{t}_\perp\|^2\big]\right) + \frac{7\alpha\gamma_x}{\lambda}  \bbE\big[\|\X^t-\underline\X^{t}\|^2\big]  \nonumber \\
&~\quad  +\frac{1}{\lambda}\left(  \textstyle \frac{1}{2}+12\alpha\gamma_x\right)\bbE\big[ \|\widehat\X^{t}-\X^{t+\frac{1}{2}}\|^2\big]   +\frac{1}{\lambda} \left( \textstyle  -\frac{1}{2}+11\alpha\gamma_x\right) \bbE\big[\|  \widehat\X^{t}- \X^{t} \|^2\big].
\nonumber \\
\leq &~ \sum_{i=1}^n\bbE\big[ \phi_\lambda( \vx_i^{t})\big] +  \frac{12}{\lambda}\bbE\big[\|\X^t_\perp\|^2\big] + \frac{7\alpha\gamma_x}{\lambda} \bbE\big[\|\X^t-\underline\X^{t}\|^2\big]  +   \frac{12}{\lambda}  \eta^2\bbE\big[\|\Y^t_\perp\|^2\big] \nonumber \\
&~+  \frac{1}{\lambda}\Big( {\textstyle\left(\frac{1}{2}+12\alpha\gamma_x \right)  \left( 1-\frac{\eta}{2\lambda} \right) + \left( -\frac{1}{2}+11\alpha\gamma_x\right) }\Big) \bbE\big[\| \widehat\X^{t}- \X^{t} \|^2\big] + \frac{5}{\lambda} \eta^2 \sigma^2,
\end{align*}
where the second inequality holds by $6\alpha\gamma_x\leq 1$, and the third inequality holds by \eqref{eq:hatx_xprox_comp} with $\frac{1}{2}+12\alpha\gamma_x\leq \frac{5}{2}$.
Noticing $$\left(\frac{1}{2}+12\alpha\gamma_x \right)  \left( 1-\frac{\eta}{2\lambda} \right) + \left( -\frac{1}{2}+11\alpha\gamma_x\right) = 23\alpha\gamma_x - \frac{\eta}{4\lambda} - \frac{6\alpha\gamma_x\eta}{\lambda}\leq 23\alpha\gamma_x - \frac{\eta}{4\lambda},$$ 
we obtain \eqref{eq:2.7} and complete the proof.
\end{proof}

With Lemmas \ref{lem:X_consensus_comperror}, \ref{lem:Y_consensus_comperror} and \ref{lem:phi_one_step}, we are ready to prove the Theorem \ref{thm:sect3thm}. We will use the Lyapunov function:
\begin{align*}
    \V^t = z_1 \bbE\big[\|\X^{t}_\perp\|^2\big] + z_2 \bbE\big[\|\X^{t}-\underline\X^{t}\|^2\big] +z_3\bbE\big[\|\Y^{t}_\perp\|^2\big]+z_4 \bbE\big[\|\Y^{t}-\underline\Y^{t}\|^2\big] + z_5  \sum_{i=1}^n \bbE[\phi_\lambda( \vx_i^{t})], 
\end{align*}
where $z_1, z_2, z_3, z_4, z_5 \geq 0$ are determined later.

\subsection*{Proof of Theorem \ref{thm:sect3thm}}

\begin{proof}
Denote 
\begin{align*} 
   &~\Omega_0^t =   \bbE[\|\widehat\X^{t}-\X^{t}\|^2], \quad \Phi^t = \sum_{i=1}^n \bbE[\phi_\lambda( \vx_i^{t})], \\
   &~ \Omega^t = \left(\bbE\big[\|\X^{t}_\perp\|^2\big], \bbE\big[\|\X^{t}-\underline\X^{t}\|^2\big], \bbE\big[\|\Y^{t}_\perp\|^2\big], \bbE\big[\|\Y^{t}-\underline\Y^{t}\|^2\big], \Phi^t\right)^\top.
\end{align*}
Then Lemmas \ref{lem:X_consensus_comperror}, \ref{lem:Y_consensus_comperror} and \ref{lem:phi_one_step} imply
$\Omega^{t+1} \leq \A\Omega^t + \vb \Omega_0^t + \vc \sigma^2$ with
\begin{align*}
    &\A = \begin{pmatrix}
             \frac{3+\widehat\rho^2_x}{4}  &~ 2\alpha\gamma_x(1-\widehat\rho_x^2) &~  \frac{9}{4(1-\widehat\rho^2_x)} \eta^2  &~ 0 &~ 0\\
             \frac{21}{1-\alpha^2}  &~  \frac{3+\alpha^2}{4} &~\frac{21}{1-\alpha^2} \eta^2 &~ 0 &~ 0 \\
             \frac{150 L^2}{1-\widehat\rho^2_y}  &~ \frac{20\sqrt{3}  L^2}{1-\widehat\rho^2_y }\alpha\gamma_x  &~  \frac{3+\widehat\rho^2_y}{4}  &~  \frac{48}{1-\widehat\rho^2_y }\alpha^2\gamma_y^2 &~ 0\\ 
             \frac{180  L^2}{1-\alpha^2}  &~ \frac{24\sqrt{3}  L^2}{1-\alpha^2} \alpha\gamma_x  &~ \frac{104\gamma_y^2+96 L^2 \eta^2}{1-\alpha^2}  &~ \frac{3+\alpha^2}{4} &~ 0\\
             \frac{12}{\lambda}  &~  \frac{7\alpha\gamma_x}{\lambda} &~  \frac{12}{\lambda}\eta^2 &~ 0 &~ 1\\
        \end{pmatrix}, \\[0.2cm] 
&\vb = 
 \begin{pmatrix}
   4\alpha\gamma_x(1-\widehat\rho_x^2) \\
     \frac{11}{1-\alpha^2}  \\
     \frac{40 L^2  }{1-\widehat\rho^2_y}\\
     \frac{48 L^2}{1-\alpha^2} \\
   \frac{1}{\lambda}\left( \textstyle -\frac{\eta}{4\lambda}   +  23\alpha\gamma_x  \right) 
 \end{pmatrix}, \quad
\vc = 
 \begin{pmatrix}
     4\alpha\gamma_x \eta^2 (1-\widehat\rho_x^2) \\
        \frac{11 \eta^2 }{1-\alpha^2} \\    
        12n \\
         \frac{10n}{1-\alpha^2}\\
         \frac{5}{\lambda} \eta^2
 \end{pmatrix}.
\end{align*}
Then for any $\vz = (z_1, z_2  , z_3, z_4, z_5 )^\top\geq \0^\top$, it holds 
\begin{align*}
    \vz^\top \Omega^{t+1}  \leq  \vz^\top \Omega^t + (\vz^\top \A-\vz^\top) \Omega^t + \vz^\top\vb \Omega_0^t + \vz^\top\vc \sigma^2.
\end{align*}
Let $\gamma_x\leq \frac{\eta}{\alpha}$ and $\gamma_y\leq \frac{(1-\alpha^2) (1-\widehat\rho^2_x)(1-\widehat\rho^2_y)}{317}$.
Take $$z_1=\frac{52}{1-\widehat\rho^2_x}, z_2 = \frac{448}{1-\alpha^2} \eta , z_3 = \frac{521}{(1-\widehat\rho^2_x)^2(1-\widehat\rho^2_y)} \eta^2, z_4=(1-\alpha^2) \eta^2, z_5=\lambda.$$ We have
\begin{align*}
\vz^\top \A-\vz^\top \leq  &~
 \begin{pmatrix}
      \frac{21\cdot448}{ (1-\alpha^2)^2} \eta +   \frac{150\cdot521 L^2\eta^2}{(1-\widehat\rho^2_x)^2(1-\widehat\rho^2_y)^2} + 180  L^2\eta^2 - 1 \\[0.2cm]
      \frac{521\cdot20\sqrt{3} L^2\eta^3}{(1-\widehat\rho^2_x)^2(1-\widehat\rho^2_y)^2} + 24\sqrt{3} L^2\eta^3 -\eta   \\[0.2cm]
     \frac{448\cdot21\eta^3}{ (1-\alpha^2)^2} + 96 L^2 \eta^4   -
       \frac{\eta^2}{(1-\widehat\rho^2_x)^2}\\[0.1cm]
    0 \\[0.1cm]
    0
     \end{pmatrix}^\top,   \\
    \vz^\top\vb \leq &~ \textstyle -\frac{\eta}{4\lambda}   +  23\eta + 48  L^2 \eta^2 + \frac{521\cdot 40 \eta^2 L^2}{(1-\widehat\rho^2_x)^2 (1-\widehat\rho^2_y)^2}  +  \frac{448\cdot11\eta}{ (1-\alpha^2)^2} +  52\cdot4 \eta,\\
\vz^\top\vc \leq &~ \left(  \textstyle 52\cdot4\eta   + \frac{448\cdot 11\eta}{ (1-\alpha^2)^2} +  \frac{521\cdot12n}{(1-\widehat\rho^2_x)^2(1-\widehat\rho^2_y)}+ 10n + 5 \right)\eta^2.
\end{align*}

By $\eta\leq \frac{(1-\alpha^2)^2(1-\widehat\rho^2_x)^2(1-\widehat\rho^2_y)^2}{18830\max\{1, L\}}$ and 
$\lambda\leq \frac{ (1-\alpha^2)^2}{9 L+41280}$,
we have $\vz^\top \A-\vz^\top \leq (-\frac{1}{2}, 0, 0, 0, 0)^\top$, 
\begin{align*}
\vz^\top\vc \leq  \textstyle 
\frac{(521\cdot12+10)n+6}{(1-\widehat\rho^2_x)^2(1-\widehat\rho^2_y)}\eta^2 =  \textstyle\frac{6262n+6}{(1-\widehat\rho^2_x)^2(1-\widehat\rho^2_y)}\eta^2
\end{align*}
and
\begin{align*}
    \vz^\top\vb ~ \leq &~    \textstyle \eta\Big(  -\frac{1}{4\lambda} + 23 + 48  L^2 \eta + \frac{521\cdot 40 \eta L^2}{(1-\widehat\rho^2_x)^2 (1-\widehat\rho^2_y)^2}  +  \frac{448\cdot11 }{ (1-\alpha^2)^2} + 52\cdot4 \Big) \nonumber \\
    \leq &~  \textstyle  -\frac{\eta}{8\lambda} + \eta\Big(  -\frac{1}{8\lambda} + \frac{ 9  L}{8 }  +  \frac{5160}{ (1-\alpha^2)^2}\Big) 
    \leq 
    -\frac{\eta}{8\lambda}.
\end{align*}
Hence we have
\begin{align}
    \vz^\top \Omega^{t+1} \leq  \textstyle  \vz^\top \Omega^{t} -\frac{\eta}{8\lambda} \Omega_0^t -\frac{1}{2}\bbE[\|\X^t_\perp\|^2] + \frac{6262n+6}{(1-\widehat\rho^2_x)^2(1-\widehat\rho^2_y)}\eta^2\sigma^2.\label{eq:l_fun_comp} 
\end{align}
Thus summing up \eqref{eq:l_fun_comp} for $t=0,1,\ldots,T-1$ gives 
\begin{align}
  \frac{1}{\lambda T}\sum_{t=0}^{T-1} \Omega_0^t +\frac{4}{\eta T}\sum_{t=0}^{T-1} \bbE[\|\X^t_\perp\|^2] \leq  \textstyle  \frac{8\left(\vz^\top \Omega^0 - \vz^\top \Omega^{T}\right)}{\eta T}  + \frac{8(6262n+6)}{(1-\widehat\rho^2_x)^2(1-\widehat\rho^2_y)} \eta\sigma^2. \label{eq:thm3_avg-Omega}
\end{align}
From $\vy_i^{-1}=\0$, $\underline\vy_i^{-1}=\0$, $\nabla F_i(\vx_i^{-1}$, $\xi_i^{-1})=\0$,  $\underline\vx_i^{0} =\0$, $\vx_i^0 = \vx^0, \forall\, i \in \hN$, we have  
\begin{gather}
    \|\Y^0_\perp\|^2  = \|\nabla \bF^0(\I-\J)\|^2\leq\|\nabla \bF^0\|^2, 
     \quad   \|\Y^{0}-\underline\Y^{0}\|^2 =  \|\nabla \bF^0-Q_\vy\big[\nabla \bF^0\big]\|^2 \leq \alpha^2 \|\nabla \bF^0\|^2, \label{eq:initial_thm3_1}\\
     \|\X^0_\perp\|^2=0, \quad \|\X^0-\underline\X^{0}\|^2=0, 
     \quad \Phi^0=n \phi_\lambda(\vx^0).
     \label{eq:initial_thm3_2}
\end{gather}
Note \eqref{eq:end_thm2} still holds here. 
With \eqref{eq:initial_thm3_1}, \eqref{eq:initial_thm3_2}, \eqref{eq:end_thm2}, and the nonnegativity of $ \bbE[\|\X^T_\perp\|^2]$, $\bbE[\|\X^{T}-\underline\X^{T}\|^2]$, $\bbE[\|\Y^T_\perp\|^2]$,  $\bbE[\|\Y^{T}-\underline\Y^{T}\|^2]$, we have
\begin{align}
\vz^\top \Omega^0 - \vz^\top \Omega^{T} \le  \textstyle
 \frac{521}{(1-\widehat\rho^2_x)^2(1-\widehat\rho^2_y)} \eta^2 \bbE[\|\nabla \bF^0\|^2] + \eta^2  \bbE[\|\nabla \bF^0\|^2] + \lambda n \phi_\lambda(\vx^0) -\lambda n  \phi_\lambda^*. \label{eq:them3_Omega0_OmegaT}
\end{align}
where we have used $\alpha^2\leq 1$ from Assumption \ref{assu:compressor}.

By the convexity of the frobenius norm and  \eqref{eq:them3_Omega0_OmegaT}, we obtain from \eqref{eq:thm3_avg-Omega} that 
\begin{align}
    &~ \frac{1}{n\lambda^2}   \bbE\big[\|\widehat\X^{\tau}-\X^{\tau}\|^2\big] +\frac{4}{n \lambda \eta} \bbE[\|\X^\tau_\perp\|^2]
   \leq  \frac{1}{n\lambda^2} \frac{1}{T}\sum_{t=0}^{T-1}  \bbE\big[\|\widehat\X^{t}-\X^{t}\|^2\big] +\frac{4}{n \lambda \eta T}\sum_{t=0}^{T-1} \bbE[\|\X^t_\perp\|^2]  \nonumber \\
   \leq  &  \textstyle \frac{8\left( \phi_\lambda(\vx^0) - \phi_\lambda^*\right)}{ \eta T}  
     +\frac{50096n+48}{(1-\widehat\rho^2_x)^2(1-\widehat\rho^2_y)} \frac{\eta}{n\lambda}\sigma^2
    \textstyle  +  \frac{8\cdot521 \eta }{n\lambda T (1-\widehat\rho^2_x)^2(1-\widehat\rho^2_y)}  \bbE\big[ \|\nabla \bF^0\|^2\big] + 
      \frac{8\eta}{n\lambda  T} \bbE\big[ \|\nabla \bF^0\|^2\big]  \nonumber \\
   \leq &~\textstyle \frac{8\left(\phi_\lambda(\vx^0) - \phi_\lambda^*\right)}{\eta T}  
     +\frac{(50096n+48)\eta \sigma^2}{n\lambda(1-\widehat\rho^2_x)^2(1-\widehat\rho^2_y)} + \textstyle \frac{4176 \eta \bbE\left[ \|\nabla \bF^0\|^2\right] }{n\lambda T (1-\widehat\rho^2_x)^2(1-\widehat\rho^2_y)}. \label{eq:them_CDProxSGT0}
\end{align}
With $\|\nabla \phi_\lambda (\vx_i^\tau)\|^2 = \frac{\|\vx_i^\tau-\widehat\vx_i^\tau\|^2}{\lambda^{2}}$ from Lemma \ref{lem:xhat_x}, we complete the proof.
\end{proof}

%% file: appendix_experiment.tex
\section{Additional Details on FixupResNet20}

FixupResNet20 \cite{zhang2019fixup}  is amended from the popular ResNet20 \cite{he2016deep} by deleting the BatchNorm layers \cite{ioffe2015batch}. The BatchNorm layers use the mean and variance of some hidden layers based on the data inputted into the models. In our experiment, the data on nodes are heterogeneous. If the models include BatchNorm layers, even all nodes have the same model parameters after training, their testing performance on the whole data would be different for different nodes because the mean and variance of the hidden layers are produced on the heterogeneous data. Thus we use FixupResNet20 instead of ResNet20. 


